\documentclass[11pt]{amsart}
\usepackage{amsthm,amsmath,amssymb}
\usepackage{stmaryrd,mathrsfs}
\usepackage{enumerate}
\usepackage[T1]{fontenc}
\usepackage{esint}
\usepackage{hyperref}
\usepackage{color}
\allowdisplaybreaks

\newcommand{\ud}[0]{\,\mathrm{d}}

\newcommand{\abs}[1]{|#1|}

\newcommand{\Norm}[2]{\|#1\|_{#2}}

\renewcommand{\Re}[0]{\operatorname{Re}}
\renewcommand{\Im}[0]{\operatorname{Im}}

\newcommand{\R}{\mathbb{R}}
\newcommand{\C}{\mathbb{C}}
\newcommand{\N}{\mathbb{N}}

\newcommand{\T}{\mathbb{T}}

\newcommand{\Exp}[0]{\mathbb{E}}

\newcommand{\eps}[0]{\varepsilon}

\swapnumbers \numberwithin{equation}{section}

\theoremstyle{plain}
\newtheorem{theorem}[equation]{Theorem}

\newtheorem{corollary}[equation]{Corollary}
\newtheorem{lemma}[equation]{Lemma}

\theoremstyle{definition}
\newtheorem{definition}[equation]{Definition}

\theoremstyle{remark}
\newtheorem{remark}[equation]{Remark}

\makeatletter
\@namedef{subjclassname@2020}{%
  \textup{2020} Mathematics Subject Classification}
 \def\@textbottom{\vskip \z@ \@plus 1pt}
 \let\@texttop\relax
\makeatother

\makeatletter
\def\namedlabel#1#2{\begingroup
   \def\@currentlabel{#2}%
   \label{#1}\endgroup
}
\makeatother

\begin{document}

\title[Bounded Holomorphic semigroups]{Quantitative estimates for bounded holomorphic semigroups}

\author[T. \ Hyt\"onen and S. Lappas]{Tuomas Hyt\"onen and Stefanos Lappas}

\address{T.H. \& S.L.: Department of Mathematics and Statistics, P.O.Box~68 (Pietari Kalmin katu~5), FI-00014 University of Helsinki, Finland}
\email{tuomas.hytonen@helsinki.fi, vlappas@hotmail.com}

\address{T.H.: Department of Mathematics and Systems Analysis, Aalto University, P.O. Box 11100, FI-00076 Aalto, Finland}
\email{tuomas.p.hytonen@aalto.fi}

\address{S.L.: Department of Mathematical Analysis, Faculty of Mathematics and Physics, Charles University, Sokolovsk\'a 83, 186 75 Praha 8, Czech Republic}
\email{stefanos.lappas@matfyz.cuni.cz}





\maketitle


\begin{abstract}
In this paper we revisit the theory of one-parameter semigroups of linear operators on Banach spaces in order to prove quantitative bounds for bounded holomorphic semigroups. Subsequently, relying on these bounds we obtain new quantitative versions of two recent results of Xu related to the vector-valued Littlewood--Paley--Stein theory for symmetric diffusion semigroups.
\end{abstract}

\section{Introduction}

The theory of Banach valued martingales involving the notions of martingale type and cotype was introduced and studied in depth by Pisier (we refer to \cite{Pisier, Pisier:1982, Pisier:1986, Pisier:2016} and Section \ref{section:4} for more information). His renorming theorem states that these are geometric properties of the underlying Banach space, characterized by the existence of an equivalent norm in the space which is uniformly convex of power type $q$ (for the precise statement of this result see Theorem \ref{thm:Pisier} in Section \ref{section:4}).

On the other hand, first Stein in \cite[Chapter IV]{S} proved the following result which 
extends the classical inequality on the Littlewood--Paley $g$-function: for every $1<p<\infty$

\begin{equation}\label{eq:LSP}
  \bigg\|\bigg(\int_0^\infty\bigg|t\frac{\partial}{\partial t}T_t f\bigg|^2\frac{dt}{t}\bigg)^{1/2}\bigg\|_{L_p(\Omega)}\approx\|f\|_{L_p(\Omega)},\quad\forall f\in L_p(\Omega),
\end{equation}
where by $\{T_t\}_{t>0}$ we denote a symmetric diffusion semigroup (see Definition \ref{def:sds} in Section \ref{section:4}) and the equivalence constants depend only on $p$.

The aforementioned theory about martingale inequalities is closely related to the vector-valued Littlewood--Paley--Stein theory which was developed in \cite{MTX, Xu:98, X}. In this direction, a Littlewood--Paley theory was first developed in \cite{Xu:98} for functions with values in uniformly convex Banach spaces. In particular, Xu obtained in \cite{Xu:98} the one-sided vector-valued extension of \eqref{eq:LSP} for the classical Poisson semigroup on the torus $\T$.

Mart\'{i}nez--Torrea--Xu \cite{MTX} characterized, in the vector-valued setting, the validity of the following one-sided inequality concerning the generalized Little\-wood--Paley $g$-function associated with a subordinated Poisson symmetric diffusion semigroup $\{P_t\}_{t>0}$ by the martingale type and cotype properties of the underlying Banach space: for given Banach space $X$ and $1<q<\infty$, $X$ is of martingale cotype $q$ iff for every $1<p<\infty$
\begin{equation}\label{eq:LSP2}
  \bigg\|\bigg(\int_0^\infty\bigg\|t\frac{\partial}{\partial t}P_t f\bigg\|_X^q\frac{dt}{t}\bigg)^{1/q}\bigg\|_{L_p(\Omega)}\lesssim\|f\|_{L_p(\Omega; X)},\quad\forall f\in L_p(\Omega; X).
\end{equation}
The converse inequality of \eqref{eq:LSP2} was also treated in \cite{MTX} iff $X$ is of martingale type $q$. Here the subordinated Poisson semigroup $\{P_t\}_{t>0}$ is defined by
\begin{equation*}
  P_t f=\frac{1}{\sqrt{\pi}}\int_0^\infty\frac{e^{-s}}{\sqrt{s}}T_{\frac{t^2}{4s}}fds,
\end{equation*}
where $\{T_t\}_{t>0}$ is a symmetric diffusion semigroup. This $\{P_t\}_{t>0}$ is again a symmetric diffusion semigroup. Recall that if $A$ denotes the negative infinitesimal generator of $\{T_t\}_{t>0}$, then $P_t=e^{-\sqrt{A}t}$.

The question whether \eqref{eq:LSP2} is also true for a general (not necessarily subordinated) diffusion semigroup $\{T_t\}_{t>0}$ in place of $\{P_t\}_{t>0}$, or even just, in the special case of the heat semigroup on $\R^n$, was left open in \cite{MTX} (see \cite[Problem 2 on p. 447]{MTX}).

This question concerning the heat semigroup on $\R^n$ was answered in the affirmative by Naor and one of us \cite{HN}, yielding good dependence of various bounds on the target Banach space $X$, but also with dependence on the dimension $n$ of the domain $\R^n$. This was further elaborated by Xu \cite{X}, who extended the previous result to general diffusion semigroups $\{T_t\}_{t>0}$ and replaced a concrete bound for the derivative of the heat kernel used in \cite{HN} by an abstract analyticity result from \cite{Pisier:1982}. While giving a form of the Littlewood--Paley--Stein inequality

\begin{equation}\label{LP}
  \bigg(\int_0^\infty\bigg\|t\frac{\partial}{\partial t} e^{t\Delta}f\bigg\|_{L^q(\R^n;X)}^q\frac{dt}{t}\bigg)^{1/q}\lesssim\|f\|_{L^q(\R^n;X)},
\end{equation}
entirely free of the domain's dimension $n$, this abstract argument relies on implicit unquantified bounds, and obscures the dependence of the bound on the target Banach space $X$, which was completely explicit in \cite{HN}.

The main result of this work is Theorem \ref{heat} in Section \ref{section:5} which provides a quantitative version of the main result of \cite[Theorem 2 (i)]{X}, including explicit dependence on the target Banach space $X$ and quantitative dimension free version of \eqref{LP} not only for the heat semigroup on $\R^n$ but also for the general diffusion semigroup $\{T_t\}_{t>0}$ (see also \cite{Xu:2021, Xu:22} for some related work).

The paper is organized as follows: in Sections \ref{section:2}, \ref{section:3} and \ref{section:4} we present the tools for proving our main result in Section \ref{section:5}. To be more precise, we provide quantitative versions of several classical results in the theory of bounded holomorphic semigroups, in particular their characterization due to Kato \cite{K}. Having these at our disposal, in Section \ref{section:4} we quantify the dependence of the analyticity bound of \cite[Remark 1.8(b)]{Pisier:1982} (see also Corollary \ref{coro:main 1}) for extensions of diffusion semigroups $\{T_t\}_{t>0}$ to uniformly convex spaces. As discussed above, this has direct implications to bounds in Littlewood--Paley--Stein inequalities \eqref{LP} by \cite{HN, X}. In addition, we provide an explicit dependence of $\sup_{t>0}\|t\partial T_{t}\|$ on the martingale cotype constant $\mathfrak{m}_{q,X}$ (see Corollary \ref{coro:main 2}), which was the case for the related estimates in \cite{HN}. Section \ref{section:5} is devoted to the proof of Theorem \ref{heat}. Finally, in Section \ref{section:6} using Corollary \ref{coro:main 2} we obtain a quantitative version of one more result of Xu \cite[Theorem 1.4]{Xu:2021} which we compare  with Theorem \ref{heat}.

\subsection*{Notation}
Throughout the paper, given $a,b\in(0,\infty)$, the notation $a\lesssim b$ means that $a\le cb$ for some universal constant $c\in(0,\infty)$. The notation $a\approx b$ stands for $(a\lesssim b)\wedge(b\lesssim a)$. If we need to allow for dependence on parameters, we indicate this by subscripts. Moreover, given a Banach space $X$, we denote by $L_p(\Omega;X)$ the usual $L_p$ space of strongly measurable functions from $\Omega$ to $X$ and we will use the abbreviation $\partial=\partial/\partial t$.

\section{Preliminaries}\label{section:2}

We begin by recalling several basic definitions and results from the semigroup theory:

\begin{definition}[Strongly Continuous Semigroup]\label{Def.1} 
Let $X$ be a Banach space. A family $(T(t))_{t\ge 0}$ of bounded linear operators on $X$ is called a {\em strongly continuous (one-parameter semigroup) (or $C_0$-semigroup)} if
\begin{enumerate}[(i)]
  \item $T(0)=I$,
  \item $T(t+s)=T(t)T(s)\quad t,s\ge 0$,
  \item\namedlabel{(iii)}{(iii)}  $\lim_{t\downarrow 0}T(t)x=x\quad\forall x\in X$.
\end{enumerate}
\end{definition}

\begin{remark}
A consequence of Definition \ref{Def.1} and the uniform boundedness principle (see also \cite[Proposition 5.5]{EN}) is the following exponential boundedness of a strongly continuous semigroup $(T(t))$: there exist constants $M\ge 1$ and $\omega\in\R$ such that
\begin{equation*}
  \|T(t)\|\leq M e^{\omega t}
\end{equation*}
for all $t\ge 0$.
\end{remark}

\begin{definition}[Infinitesimal Generator]
Let $(T(t))$ be a $C_0$-semigroup. The {\em infinitesimal generator} of $(T(t))$ is the linear operator defined by
\begin{equation*}
\begin{split}
  D(A)&:=\bigg\{x\in X:\lim_{h\downarrow 0}\frac{T(h)x-x}{h}\;\text{exists}\bigg\},  \\ Ax&:=\lim_{h\downarrow 0}\frac{T(h)x-x}{h}.
\end{split}
\end{equation*}
\end{definition}

\begin{definition}
Let $A:D(A)\subset X\rightarrow X$ be a a closed, linear operator on some Banach space $X$.
We call
\begin{equation*}
  \rho(A):=\{\lambda\in\C:\lambda-A:D(A)\rightarrow X\;\;\text{is bijective}\}
\end{equation*}
the {\em resolvent set} and its complement $\sigma(A):=\C\setminus\rho(A)$ the {\em spectrum} of $A$. For $\lambda\in\rho(A)$, the inverse
\begin{equation*}
  R(\lambda,A):=(\lambda-A)^{-1}
\end{equation*}
is, by the closed graph theorem, a bounded operator on $X$ and will be called the {\em resolvent} of $A$ at the point $\lambda$.
\end{definition}

\begin{theorem}[\cite{EN}, Theorem II.3.8 or \cite{P}, Theorem I.5.3 and Remark I.5.4](Hille--Yosida Generation Theorem)\label{thm:Hille-Yosida}
Let $A$ be a linear operator on $X$ and let $M\ge 1, \omega\in\R$. The following conditions are equivalent:
\begin{enumerate}[(a)]
  \item $A$ generates a $C_0$-semigroup satisfying
  \begin{equation*}
    \|T(t)\|\leq M e^{\omega t}\qquad\text{for}\;t\ge 0.    
  \end{equation*}
  \item $A$ is closed and densely defined, the resolvent set $\rho(A)$ contains the half-plane $\{\lambda\in\C:\Re\lambda>\omega\}$, and
  \begin{equation*}
    \|R(\lambda, A)^n\|\leq\frac{M}{(\Re\lambda-\omega)^n}\qquad\text{for all}\;n\in\N.
  \end{equation*}
\end{enumerate}
\end{theorem}

\begin{definition}[Open/Closed Sector]
Let $\alpha\in[0,\pi)$. We call
\begin{equation*}
\begin{split}
  &\Sigma_{\alpha}:=\{z\in\C\setminus\{0\}:|\arg(z)|<\alpha\},  \\
  &\overline{\Sigma_{\alpha}}:=\{z\in\C:|\arg(z)|\le\alpha\} 
\end{split}
\end{equation*}
{\em the open/closed sector of angle $\alpha$ in the complex plane}; the argument is taken in $(-\pi,\pi)$. 
\end{definition}

\begin{definition}[Holomorphic Semigroup]\label{def:hol.semi.}
A family of bounded linear operators $(T(z))_{z\in\Sigma_{\delta}\cup\{0\}}$ is called a {\em holomorphic semigroup (of angle $\delta$)} if
\begin{enumerate}[(i)]
  \item\namedlabel{prop (i)}{(i)} $T(z_1+z_2)=T(z_1)T(z_2)$\;\text{for all}\;$z_1,z_2\in \Sigma_{\delta}$.
  \item\namedlabel{prop (ii)}{(ii)} The map $z\mapsto T(z)$ is holomorphic in $\Sigma_{\delta}$.
  \item\namedlabel{prop (iii)}{(iii)} $\displaystyle{\lim_{\substack{z\to 0\\ z\in \overline{\Sigma}_{\delta'}}}T(z)x=x}$ for all $x\in X$ and $0<\delta'<\delta$. 
\end{enumerate}  
If, in addition,  
\begin{enumerate}[(iv)]  
  \item $\|T(z)\|$ is bounded in $\overline{\Sigma}_{\delta'}$ for every $0<\delta'<\delta$, 
\end{enumerate}
we call $(T(z))_{z\in\Sigma_{\delta}\cup\{0\}}$ a {\em bounded holomorphic semigroup}.
\end{definition}

\begin{theorem}[\cite{EN}, IV.1.2](Resolvent equation)\label{thm:resol.eq.} 
Let $A$ be a closed linear operator. For $\lambda,\mu\in \rho(A)$ one has
\begin{equation*}
  R(\lambda,A)-R(\mu,A)=(\mu-\lambda)R(\lambda,A)R(\mu,A).
\end{equation*}
\end{theorem}


\section{Classical semigroup theory made quantitative}\label{section:3}

In this section we collect the additional results from semigroup theory that we will use in Section \ref{section:4}. Most of these results are in principle known, but difficult to find in the precise quantitative form that we need. For this reason, we also revisit most of the proofs in order to track the dependence on the various constants. Besides our application in Section \ref{section:4}, we hope that recording these quantitative formulations might have an independent interest and possible other applications elsewhere. In addition to the standard references \cite{EN, P}, we have benefited from the nice presentation of the classical theory in \cite{F}.

\begin{lemma}[\cite{EN}, Theorem 4.6; \cite{P}, Chapter 2, Theorem 5.2; \cite{F}, Theorem 1.1.23]\label{lem:main1}
Let $(A,D(A))$ be a linear operator on $X$. Suppose that $A$ generates a bounded strongly continuous semigroup $(T(t))$ with $\|T(t)\|\le M$ on $X$, where $M\ge1$, and there exists a constant $C>0$ such that
\begin{equation}\label{eqqq1}
  \|R(r+is,A)\|\le\frac{C}{|s|}
\end{equation}
for all $r>0$ and $0\ne s\in\R$. Then $\Sigma_{\frac{\pi}{2}+\delta}\subset\rho(A)$ for $\delta=\arctan(\frac{1}{C})$ and the resolvent of $A$ satisfies 
\begin{equation}\label{eqqq2}
  \|\lambda R(\lambda,A)\|\leq\frac{\sqrt{C^2+M^2}}{1-q}
\end{equation}
for all $q\in(0,1)$ and all $\lambda\in\C$ such that $|\arg(\lambda)|\le\frac{\pi}{2}+\arctan(\frac{q}{C})$.
\end{lemma}

\begin{proof}
This is essentially contained in the proof of (b) $\Rightarrow$ (a) of \cite[Theorem 1.1.23]{F}, so we only indicate minor modifications to reach the quantitative version as stated.

As in \cite[page 21]{F}, the Hille--Yosida Theorem \ref{thm:Hille-Yosida} gives the bound
\begin{equation*}
  \|R(\lambda,A)\|\le\frac{M}{\Re\lambda}\quad\text{for}\quad\Re\lambda>0.
\end{equation*}
Combined with assumption \eqref{eqqq1}, it follows that
\begin{equation*}
  \|R(\lambda,A)\|\le\min\bigg\{\frac{M}{\Re\lambda},\frac{C}{|\Im\lambda|}\bigg\}.
\end{equation*}
Now either $\Re\lambda\ge\frac{M}{\sqrt{M^2+C^2}}|\lambda|$ or $|\Im\lambda|\ge\frac{C}{\sqrt{M^2+C^2}}|\lambda|$ holds and therefore
\begin{equation}\label{eq:bdd. resolv. 1}
  \|R(\lambda,A)\|\le\frac{\sqrt{C^2+M^2}}{|\lambda|}\quad\text{for}\;\Re\lambda>0.
\end{equation}
(The argument in \cite[page 21]{F} is similar but gives the slightly larger constant $\sqrt{2}(C+M)$ in place of $\sqrt{C^2+M^2}$.)

Let then $q\in(0,1)$ and 
\begin{equation*}\tag{$*$}
  \Re\lambda\le0,\qquad |\arg(\lambda)|\le\frac{\pi}{2}+\arctan(\frac{q}{C})
\end{equation*}
thus $|\Re\lambda|/|\Im\lambda|\leq q/C$. For any $q'\in(q,1)$, \cite[page 22]{F} obtains the estimates
\begin{equation*}
    \|R(\lambda,A)\|\le\frac{C}{1-q'}\frac{1}{|\Im\lambda|},\qquad
    \frac{1}{|\Im\lambda|^2}\leq\Big(\frac{q^2}{C^2}+1\Big)\frac{1}{|\lambda|^2}.
\end{equation*}
Taking the limit $q'\to q$, we conclude that
\begin{equation}\label{eq:bdd. resolv. 2}
  \|R(\lambda,A)\|\le\frac{C}{1-q}\frac{\sqrt{q^2+C^2}}{C}\frac{1}{|\lambda|}\le\frac{\sqrt{C^2+M^2}}{1-q}\frac{1}{|\lambda|}
\end{equation}
for all $\lambda$ as in $(*)$.
By combining estimates \eqref{eq:bdd. resolv. 1} and \eqref{eq:bdd. resolv. 2} we deduce that
\begin{equation*}
  \|\lambda R(\lambda,A)\|\leq\frac{\sqrt{C^2+M^2}}{1-q}
\end{equation*}
for all $q\in(0,1)$ and all $\lambda\in\C$ such that $|\arg(\lambda)|\le\frac{\pi}{2}+\arctan(\frac{q}{C})$.
\end{proof}

For linear operators satisfying the conclusions of Lemma \ref{lem:main1} and appropriate paths $\gamma$, the exponential function ``$e^{tA}$'' can now be defined via the Cauchy integral formula. In particular, we write
\begin{equation}\label{eq:Cauchy int.}
  T(z):=e^{zA}=\frac{1}{2\pi i}\int_{\gamma}e^{\mu z}R(\mu,A)\;d\mu,
\end{equation}
where $z\in\Sigma_{\delta}$ and $\gamma$ is a smooth curve in $\Sigma_{\frac{\pi}{2}+\delta}\subset\rho(A)$ running from $\infty e^{-i(\pi/2+\delta')}$ to $\infty e^{i(\pi/2+\delta')}$ for some $\delta'\in(|\arg(z)|,\delta)$.

\begin{lemma}[\cite{EN}, Proposition 4.3 and \cite{P}, Chapter 2, Theorem 5.2]\label{lem:main2}
Let $(A,D(A))$ be a linear operator in $X$ such that $\Sigma_{\frac{\pi}{2}+\delta}\subset\rho(A)$ for \; $\delta=\arctan(\frac{1}{C})$, where $C\ge 1$, and the resolvent of $A$ satisfies
\begin{equation}\label{eqqqq2}
  \|\mu R(\mu,A)\|\leq\frac{\sqrt{C^2+M^2}}{1-q}
\end{equation}
for all $q\in(0,1)$, some $M\ge 1$ and all $\mu\in\C$ such that $|\arg(\mu)|\le\frac{\pi}{2}+\arctan(\frac{q}{C})$. Then, for all $z\in\Sigma_{\delta}$, the maps $T(z)$ are bounded linear operators on $X$ and
\begin{equation*}
  \|T(z)\|\leq 8\frac{\sqrt{C^2+M^2}}{1-u}\bigg(1+\log\bigg(\frac{C}{1-u}\bigg)\bigg),
\end{equation*}
for all $|\arg(z)|\le\arctan(\frac{u}{C})$ and all $u\in(0,1)$. 
\end{lemma}

\begin{proof}
Let $z\in\Sigma_{\frac{\pi}{2}+\delta}$ such that $|\arg(z)|\le\arctan(\frac{u}{C})$ for some $u\in(0,1)$. We consider an auxiliary number $q\in(u,1)$, to be chosen later.

We will verify that the integral in \eqref{eq:Cauchy int.} defining $T(z)$ converges uniformly in $\mathcal{L}(X)$ with respect to the operator norm. 
Since the functions $e^{\mu z}$ and $R(\mu,A)$ are analytic in $\Sigma_{\frac{\pi}{2}+\delta}$, this integral, if it exists, is by Cauchy's integral theorem independent of the particular choice of $\gamma$. Hence,
we may choose $\gamma$ to be the positively oriented boundary of $\Sigma_{\frac{\pi}{2}+\arctan(\frac{q}{C})}\setminus B_{\frac{1}{|z|}}(0)$. Now, we decompose $\gamma$ in three parts given by the arc $\widetilde{\gamma}$ formed by the boundary of the disk and the two rays going to infinity (this follows by change of variables $\mu=re^{\pm i(\frac{\pi}{2}+\arctan(\frac{q}{C}))}$):
\begin{equation}\label{eq*4}
\begin{split}
  \frac{1}{2\pi i}&\int_{\gamma}e^{\mu z}R(\mu,A)\;d\mu=\frac{1}{2\pi i}\int_{\widetilde{\gamma}}e^{\mu z}R(\mu,A)\;d\mu  \\
  &+\frac{1}{2\pi i}\int_{\frac{1}{|z|}}^{\infty}e^{re^{i(\frac{\pi}{2}+\arctan(\frac{q}{C}))}z}R(re^{i(\frac{\pi}{2}+\arctan(\frac{q}{C}))},A)e^{i(\frac{\pi}{2}+\arctan(\frac{q}{C}))}\;dr  \\
  &-\frac{1}{2\pi i}\int_{\frac{1}{|z|}}^{\infty}e^{re^{-i(\frac{\pi}{2}+\arctan(\frac{q}{C}))}z}R(re^{-i(\frac{\pi}{2}+\arctan(\frac{q}{C}))},A)e^{-i(\frac{\pi}{2}+\arctan(\frac{q}{C}))}\;dr.
\end{split}
\end{equation}

Since $0<u<q<1\leq C$, we have $|\arg(z)|\leq\arctan(\frac{u}{C})<\arctan(\frac{q}{C})<\arctan(1)=\frac{\pi}{4}$ and thus
\begin{equation*}
  \frac{\pi}{2}+\arctan\bigg(\frac{q}{C}\bigg)\pm\arg(z)\in\bigg(\frac{\pi}{2},\pi\bigg).
\end{equation*}
Hence,
\begin{equation*}
\begin{split}
  \Re(re^{\pm i(\frac{\pi}{2}+\arctan(\frac{q}{C}))}z)&=\Re(re^{\pm i(\frac{\pi}{2}+\arctan(\frac{q}{C}))}|z|e^{i\arg(z)})  \\
  &= r|z|\cos\bigg(\frac{\pi}{2}+\arctan\bigg(\frac{q}{C}\bigg)+\arg(z)\bigg)  \\
  &\le r|z|\cos\bigg(\frac{\pi}{2}+\arctan\bigg(\frac{q}{C}\bigg)-\arctan\bigg(\frac{u}{C}\bigg)\bigg)  \\
  &=-r|z|\sin\bigg(\arctan\bigg(\frac{q}{C}\bigg)-\arctan\bigg(\frac{u}{C}\bigg)\bigg),
\end{split}
\end{equation*}
where
\begin{equation*}
  \arctan\bigg(\frac{q}{C}\bigg)-\arctan\bigg(\frac{u}{C}\bigg)
  =\int_{\frac{u}{C}}^{\frac{q}{C}}\frac{dx}{1+x^2}
  \geq\int_{\frac{u}{C}}^{\frac{q}{C}}dx =\frac{q-u}{C}.
\end{equation*}
Since $\sin x\geq\frac{2}{\pi}x$ for $x\in[0,\frac{\pi}{2}]$, it further follows that
\begin{equation*}
  \Re(re^{\pm i(\frac{\pi}{2}+\arctan(\frac{q}{C}))}z)
  \leq -r|z|\eps',\qquad\eps':= \frac{2(q-u)}{\pi C}>0.
\end{equation*}
Combining this with \eqref{eqqqq2} and \eqref{eq*4}, and denoting $\widetilde M:=\frac{\sqrt{C^2+M^2}}{1-q}$, we get
\begin{align}\label{final ineq. 1}
  \bigg\|\frac{1}{2\pi i}\int_{\gamma}e^{\mu z}R(\mu,A)\;d\mu\bigg\| &\le\frac{\widetilde{M}}{2\pi}\int_{\widetilde{\gamma}}e^{\Re(\mu z)}\frac{|d\mu|}{|\mu|}
 +\frac{\widetilde{M}}{\pi}\int_{\frac{1}{|z|}}^{\infty}e^{-\eps'r|z|}\frac{dr}{r} \nonumber \\
  &\le\frac{\widetilde{M}|z|}{2\pi}\int_{\widetilde{\gamma}}e^{|z^{-1}z|}\;|d\mu|
  +\frac{\widetilde{M}}{\pi}\int_{\eps'}^{\infty}e^{-s}\frac{ds}{s} \nonumber \\
  &\le e\widetilde{M}+\frac{\widetilde{M}}{\pi}\Big(\int_{\eps'}^1 \frac{ds}{s}+\int_{1}^{\infty}e^{-s}ds\Big) \nonumber \\
  &= e\widetilde{M}+\frac{\widetilde{M}}{\pi}\Big(\log\frac{1}{\eps'}+e^{-1}\Big)  \nonumber \\
  &= \widetilde M\bigg(e+\frac{1}{e\pi}\bigg)+\frac{\widetilde M}{\pi}\log\bigg(\frac{1}{\eps'}\bigg) \nonumber  \\
  &\leq\widetilde M\bigg(3+\frac{1}{\pi}\log\bigg(\frac{1}{\eps'}\bigg)\bigg).
\end{align}
We now choose $q=\frac{1+Cu}{1+C}\in(u,1)$, so that
\begin{equation*}
  1-q=\frac{C(1-u)}{1+C},\qquad q-u=\frac{1-u}{1+C}.
\end{equation*}
Recalling that $C\geq 1$, it follows that
\begin{equation*}
  \widetilde M=\frac{\sqrt{C^2+M^2}}{1-q}=\frac{1+C}{C}\frac{\sqrt{C^2+M^2}}{1-u}\leq 2\frac{\sqrt{C^2+M^2}}{1-u}
\end{equation*}
and
\begin{equation*}
\begin{split}
  \log\frac{1}{\eps'}
  &=\log\frac{\pi C}{2(q-u)}
  =\log\frac{\pi C(1+C)}{2(1-u)} \\
  &=\log\pi+\log C+\log\frac{1+C}{2}+\log\frac{1}{1-u},
\end{split}
\end{equation*}
where $\log\frac{1+C}{2}\leq \log C$. Hence
\begin{equation*}
  3+\frac{1}{\pi}\log\bigg(\frac{1}{\eps'}\bigg)
  \leq 3+\frac{\log\pi}{\pi}+\frac{2}{\pi}\Big(\log C+\log\frac{1}{1-u}\Big)
  \leq 4+\log\frac{C}{1-u}
\end{equation*}
and finally
\begin{equation*}
  \widetilde M\bigg(3+\frac{1}{\pi}\log\bigg(\frac{1}{\eps'}\bigg)\bigg)
  \leq 8\frac{\sqrt{C^2+M^2}}{1-u}\Big(1+\log\frac{C}{1-u}\Big).
\end{equation*}
Together with \eqref{final ineq. 1}, this shows that
This shows that the integral defining $(T(z))$ converges in $\mathcal{L}(X)$ absolutely and uniformly for every $z\in\C$ such that $|\arg(z)|\le\arctan(\frac{u}{C})$, and proves the asserted bound.
\end{proof}

\begin{remark}\label{rmk:main1}
As shown in \cite[Proposition 4.3]{EN} (see also \cite[Chapter 2, Theorem 5.2]{P}) the previous family of bounded linear operators $T(z)$ also satisfies properties \ref{prop (i)}, \ref{prop (ii)} and \ref{prop (iii)} of Definition \ref{def:hol.semi.}.
\end{remark}

The following is a variant of \cite[Corollary 3.7.12]{ABHN} and \cite[Chapter 2, Theorem 5.2]{P}; except for the quantitative bound, it is stated in the same form in \cite[Lemma 1.1.28]{F}:

\begin{theorem}\label{thm:hol.semi. 1}
Let $A$ be the infinitesimal generator of a strongly continous semigroup $T(t)$ with $\|T(t)\|\le M$, $M\ge1$. Suppose that there exists a constant $C\ge1$ such that $is\in\rho(A)$ for $|s|>0$ and 
\begin{equation}\label{eqq1}
  \|R(is,A)\|\le\frac{C}{|s|}.
\end{equation}
Then $A$ generates a bounded holomorphic semigroup $T(z)$, defined at every $z\in\C$ such that $|\arg(z)|\lesssim\frac{1}{CM}$, and satisfying the bound
\begin{equation*}
  \|T(z)\|\lesssim CM(1+\log(CM)).
\end{equation*}
\end{theorem}

\begin{proof}
The proof follows the idea of the proof of \cite[Lemma 1.1.28]{F}. Since a self-contained proof is not much longer than an indication of the relevant changes, we give it for the reader's convenience.

The Hille--Yosida Theorem \ref{thm:Hille-Yosida} shows that
\begin{equation}\label{eqq2}
  \lambda\in\rho(A),\quad\|R(\lambda,A)\|\le\frac{M}{\Re\lambda}\quad\text{for all }\lambda\in\C\text{ with }\Re\lambda>0.
\end{equation}
If $r>0$ and $s\ne 0$, the resolvent equation (Theorem \ref{thm:resol.eq.}) provides the identity
\begin{equation*}
  R(r+is,A)=(is-(r+is))R(r+is,A)R(is,A)+R(is,A).
\end{equation*}
Taking the operator norm on both sides and applying \eqref{eqq1} and \eqref{eqq2}, we deduce the bound
\begin{equation*}
\begin{split}
  \|R(r+is,A)\| &\le Cr\|R(r+is,A)\||s|^{-1}+C|s|^{-1} \\
  &\le CM|s|^{-1}+C|s|^{-1}=C(M+1)|s|^{-1}=\widetilde C|s|^{-1}.
\end{split}
\end{equation*}
where $\widetilde C:=C(M+1)$. Hence, $A$ satisfies the assumptions and therefore the conclusions of Lemma \ref{lem:main1}. Moreover, by Lemma \ref{lem:main2} we obtain
\begin{equation*}
\begin{split}
  \|T(z)\|&\lesssim\frac{\sqrt{\widetilde C^2+M^2}}{1-u}\bigg(1+\log\bigg(\frac{\widetilde C}{1-u}\bigg)\bigg)  \\
  &\lesssim(\widetilde{C}+M)(1+\log(\widetilde{C}))\quad\text{for}\quad |\arg(z)|\lesssim\frac{1}{\widetilde{C}},
\end{split}
\end{equation*}
where in the last step we chose $u=\frac{1}{2}$. Therefore $A$ generates a bounded holomoprhic semigroup with the stated properties by Remark \ref{rmk:main1}.
\end{proof}

The following is a quantitative version of a theorem of Kato \cite{K}. We follow the proof given in \cite[Lemma 1.2.2]{F}:

\begin{theorem}\label{lem:hol.semi. 2}
Let $(T(t))$ be a strongly continuous semigroup with $\|T(t)\|\le M$, $M\ge 1$. Suppose that there is a complex number $\zeta$ with $|\zeta|=1$ and $K\in(0,\infty)$ such that
\begin{equation}\label{eq1}
  \|(\zeta I-T(t))x\|\ge \|x\|/K\qquad\text{for}\;x\in X\;\text{and all}\;t>0.
\end{equation}
Then $(T(t))$ can be extended to a bounded holomorphic semigroup. In particular,
\begin{equation*}
  \|T(z)\|\lesssim\widetilde C(1+\log(\widetilde C)),
\end{equation*}
for every $z\in\C$ such that $|\arg(z)|\lesssim\frac{1}{\widetilde{C}}$, where $\widetilde{C}:=M^2 K>0$.
\end{theorem}

Note that any constant $K\in(0,\infty)$ in \eqref{eq1} must necessarily satisfy $K\geq\frac12$.  Indeed, as $t\to 0$, we have $T(t)x\to x$, and hence
\begin{equation*}
 \|x\|/K\leq \|(\zeta I-T(t))x\| \leq |\zeta| \|x\|+\|T(t)x\|\to 2\|x\|.
\end{equation*}

\begin{proof}
We will borrow the results of some intermediate steps of the proof of the implication (c) $\Rightarrow$ (a) of \cite[Lemma 1.2.2]{F}. That proof is set up for a more general semigroup with a growth bound $\|T(t)\|\le Me^{\omega t}$, so we can simply take $\omega=0$. Moreover, \eqref{eq1} is only assumed for $t\in(0,\delta)$, so we can take $\delta=\infty$, thus $1/\delta=0$. This will simplify some of the formulas of \cite{F}.

As in \cite{F}, we take $\theta_1\in[0,2\pi)$ and $\theta_2:=2\pi-\theta_1\in(0,2\pi]$ such that $e^{i\theta_1}=e^{-i\theta_2}=\zeta$, where $\zeta$ is the number appearing in \eqref{eq1}. In \cite[page 30]{F}, one then considers $\alpha>\max\{\theta_1/\delta,\omega\theta_1\}$. In our case, this becomes simply $\alpha>0$.
In \cite[page 31]{F}, one further chooses $\eps>0$, $\xi_0>\omega$, and $\alpha_1$ such that 
\begin{equation*}
  \frac{\alpha_1-\omega\theta_1}{MK\theta_1}-\xi_0>\eps.
\end{equation*}
For $\omega=0$, we can take $\alpha_1>0$ as small as we like, then $\eps,\xi_0\in(0,\frac{\alpha_1}{2MK\theta_1})$, which clearly satisfy the requirements. The second-to-last paragraph of \cite[page 31]{F} then concludes that
\begin{equation*}
  i\alpha\in\rho(A),\quad \| R(i\alpha,A) \|\leq \frac{MK\theta_1}{\alpha}
\end{equation*}
for all $\alpha>\alpha_1$. Since we can take $\alpha_1>0$ as small as we like, this holds for every $\alpha>0$. As in \cite[page 31]{F}, the same argument for $-\alpha$ in place of $\alpha$, and $\theta_2$ in place of $\theta_1$, leads to
\begin{equation*}
  -i\alpha\in\rho(A),\quad \| R(-i\alpha,A) \|\leq \frac{MK\theta_2}{\alpha},\qquad\text{for all}\quad\alpha>0.
\end{equation*}
A combination of the last two displays then shows that
\begin{equation*}
  i\alpha\in\rho(A),\quad \| R(i\alpha,A) \|\leq \frac{C}{|\alpha|},\qquad\text{for all}\quad\alpha\in\R\setminus\{0\},
\end{equation*}
where $C:=MK\max\{\theta_1,\theta_2\}$. Note that
\begin{equation*}
  \pi= \frac12(\theta_1+\theta_2)\leq  \max\{\theta_1,\theta_2\} \leq\theta_1+\theta_2=2\pi.
\end{equation*}
Hence the assumptions of Theorem \ref{thm:hol.semi. 1} are satisfied with $M$ and $C$ as just defined, and the said theorem implies that $(T(t))$ can be extended to a bounded holomorphic semigroup with the estimate
\begin{equation*}
\begin{split}
  \|T(z)\|\lesssim \widetilde C(1+\log(\widetilde C)),
\end{split}
\end{equation*}
for every $z\in\C$ such that $|\arg(z)|\lesssim\frac{1}{\widetilde{C}}$, where $\widetilde{C}:=M^2 K>0$.
\end{proof}

\begin{remark}
A version of Theorem \ref{lem:hol.semi. 2} is also true (see \cite[Lemma 1.2.2]{F} and \cite{K}), where \eqref{eq1} is only assumed for $0<t<\delta$, but we only treat the case stated, since it simplifies the quantitative conclusions and it is the only case that we need for our applications.
\end{remark}

\begin{corollary}\label{coro:main}
Let $(T(t))$ be a strongly continuous semigroup with $\|T(t)\|\\\le M$, for some $M\ge 1$. In addition, suppose that there exists $0<\eps\le 2$ such that
\begin{equation}\label{eq:one of main}
  \|T(t)-I\|\le 2-\eps\qquad\text{for all}\;t>0.
\end{equation}
Then $(T(t))$ can be extended to a bounded holomorphic semigroup in $\Sigma_\theta$, where $\theta\approx\frac{1}{\widetilde{C}}$ and $\widetilde{C}=\frac{M^2}{\eps}>0$. In particular, 
\begin{equation}\label{eq:main bdd. 1}
  \|T(z)\|\lesssim\frac{M^2}{\eps}\bigg(1+\log\bigg(\frac{M}{\eps}\bigg)\bigg),
\end{equation}
for every $z\in\C$ such that $|\arg(z)|\lesssim\frac{1}{\widetilde{C}}$ and
\begin{equation}\label{eq:main bdd.}
  \sup_{t>0}\|t T'(t)\|\lesssim\frac{M^4}{\eps^2}\bigg(1+\log\bigg(\frac{M}{\eps}\bigg)\bigg).
\end{equation}
\end{corollary}

\begin{proof}
We choose $\zeta=-1$. Since
\begin{equation*}
  \|(-I-T(t))x\|=\|-2x+(I-T(t))x\|\ge 2\|x\|-\|(I-T(t))x\|\ge\eps\|x\|,
\end{equation*}
we deduce from Theorem \ref{lem:hol.semi. 2} with $K=\frac{1}{\eps}\ge\frac{1}{2}$ that $(T(t))$ extends to a bounded holomorphic semigroup which satisfies the following estimate:
\begin{equation}\label{eq:main bddd. 1}
  \|T(z)\|\lesssim\frac{M^2}{\eps}\bigg(1+\log\bigg(\frac{M}{\eps}\bigg)\bigg),
\end{equation}
for every $z\in\C$ such that $|\arg(z)|\lesssim\frac{1}{\widetilde{C}}$, where $\widetilde{C}=\frac{M^2}{\eps}>0$.
By Cauchy's integral formula we get
\begin{equation}\label{eq:main bdd. 2}
  T'(t)=\frac{1}{2\pi i}\oint\limits_{\gamma}\frac{T(z)}{(t-z)^2}\;dz,
\end{equation}
where $\gamma$ is the boundary of the disk $B_{t \sin(\theta)}(t)$ for $t>0$ and $\theta\approx\frac{1}{\widetilde{C}}$ is the angle of the sector $\Sigma_{\theta}$ in which $T(z)$ is holomorphic. Using \eqref{eq:main bddd. 1} and \eqref{eq:main bdd. 2} we deduce that
\begin{equation*}
\begin{split}
  \|T'(t)\|&\lesssim\frac{1}{2\pi}\frac{M^2}{\eps}\bigg(1+\log\bigg(\frac{M}{\eps}\bigg)\bigg)\max_{z:|t-z|=t\sin{\theta}}\bigg|\frac{1}{(t-z)^2}\bigg|2\pi|t\sin(\theta)|  \\
  &=\frac{1}{t\sin(\theta)}\frac{M^2}{\eps}\bigg(1+\log\bigg(\frac{M}{\eps}\bigg)\bigg)
\end{split}  
\end{equation*}
Hence,
\begin{equation*}
\begin{split}
  \sup_{t>0}\|t T'(t)\|&\lesssim\frac{1}{\sin(\theta)}\frac{M^2}{\eps}\bigg(1+\log\bigg(\frac{M}{\eps}\bigg)\bigg)  \\
 &\lesssim \frac{M^4}{\eps^2}\bigg(1+\log\bigg(\frac{M}{\eps}\bigg)\bigg),
\end{split}
\end{equation*}
where in the last step $\sin(\theta)\approx\frac{1}{\widetilde{C}}$ if $\theta\approx\frac{1}{\widetilde{C}}$.
\end{proof}

\begin{remark}
The first part of Corollary \ref{coro:main} was proved by Kato in \cite{K} under the assumption of \eqref{eq:one of main} for $0<t<\delta$. The novelty here is the estimate \eqref{eq:main bdd. 1} and \eqref{eq:main bdd.}.
\end{remark}

\section{Quantitative analyticity of diffusion semigroups in uniformly convex spaces}\label{section:4}
In our discussions below we follow \cite{X} but we provide new details concerning the quantitative bounds. Let us recall the following definitions about symmetric diffusion semigroups, uniform convexity, martingale type and cotype:

\begin{definition}\label{def:sds}
Let $(\Omega,\mathcal{A},\mu)$ be a $\sigma$-finite measure space. An operator $T$ on $(\Omega,\mathcal{A},\mu)$ is a {\em symmetric Markovian operator} if it satisfies the following properties:
\begin{enumerate}
  \item[(1)] $T$ is a linear contraction on $L_p(\Omega)$ for every $1\le p\le\infty$;
  \item[(2)] $T$ is positivity preserving and $T1=1$; and  
  \item[(3)] $T$ is a self-adjoint operator on $L_2(\Omega)$.
\end{enumerate}
\end{definition}

\begin{definition}
Let $(\Omega,\mathcal{A},\mu)$ be a $\sigma$-finite measure space. By a {\em symmetric diffusion semigroup} on $(\Omega,\mathcal{A},\mu)$ in Stein's sense \cite[Section III.1]{S}, we mean a family $\{T_t\}_{t>0}$ of linear maps where each member of this family is a symmetric Markovian operator and satisfies the following two additional properties:
\begin{enumerate}
  \item[(4)] $T_t T_s=T_{t+s}$; 
  \item[(5)] $\lim_{t\rightarrow 0}T_t f=f$ in $L_2(\Omega)$ for every $f\in L_2(\Omega)$.
\end{enumerate}
\end{definition}

\begin{definition}
A Banach space $X$ is said to be uniformly convex if for every $\eps\in(0,2]$ there exists $\delta\in(0,1]$ such that, for any vectors $x,y$ in the closed unit ball (i.e. $\|x\|\le 1$ and $\|y\|\le 1$) with $\|x-y\|\ge\eps$, one has $\|x+y\|\le 2(1-\delta)$. In addition, $(X,\|\cdot\|)$ is said to be uniformly convex of power type $q$ with $2\le q<\infty$ and positive constant $\delta$ if the following inequality holds:
\begin{equation}\label{eq:Pisier}
  \bigg\|\frac{x+y}{2}\bigg\|^q+\delta\bigg\|\frac{x-y}{2}\bigg\|^q\le\frac{1}{2}(\|x\|^q+\|y\|^q),\qquad\forall x,y\in X.
\end{equation}
\end{definition}

Examples of uniformly convex spaces that satisfy \eqref{eq:Pisier} with $\delta=1$ are the following:
\begin{enumerate}[(1)]
  \item $X$ is a Hilbert space and $q=2$. By the Parallelogram law for inner product spaces it follows:
  \begin{equation*}
  \bigg\|\frac{x+y}{2}\bigg\|^2+\bigg\|\frac{x-y}{2}\bigg\|^2=\frac{1}{2}(\|x\|^2+\|y\|^2),\qquad\forall x,y\in X.
  \end{equation*}
  \item $X=L_q(\Omega)$ and  $2\le q<\infty$. We have the following Clarkson's inequality (see \cite[Corollary 2.29]{HNVW}):
  \begin{equation*}
  \bigg\|\frac{x+y}{2}\bigg\|_{L_q}^q+\bigg\|\frac{x-y}{2}\bigg\|_{L_q}^q\le\frac{1}{2}(\|x\|_{L_q}^q+\|y\|_{L_q}^q),\qquad\forall x,y\in X. 
  \end{equation*}
\end{enumerate}

\begin{definition}\label{Def:mtc}
Let $1<q<\infty$. A Banach space $X$ is of {\em martingale cotype $q$} if there exists a positive constant $C$ such that every finite $X$ valued $L_q$-martingale $(f_n)$ defined on some probability space satisfies the following inequality:
\begin{equation*}
  \sum_{n\ge1}\Exp\|f_n-f_{n-1}\|^q_X\le C^q\sup_{n}\Exp\|f_n\|^q_X,
\end{equation*}
where $\Exp$ denotes the expectation on the underlying probability space. We then must have $q\ge 2$. $X$ is of {\em martingale type $q$} if the reverse inequality holds.
\end{definition}

\begin{remark}
Following the notation in \cite{Pisier:1986}, given a Banach space $(X,\|\cdot\|)$ and $q\ge2$ we will denote the martingale cotype $q$ constant of $X$ by $\mathfrak{m}_{q,X}\in [1,\infty]$.
\end{remark}

\begin{theorem}[\cite{Pisier}, Theorem 3.1]\label{thm:Pisier}
Let $q$ be such that $2\le q<\infty$ and let $X$ be a Banach space.
Assume that $X$ is of martingale cotype $q$, namely:
\begin{equation*}
  \sum_{n\ge1}\Exp\|f_n-f_{n-1}\|^q_X\le \mathfrak{m}_{q,X}^q\sup_{n}\Exp\|f_n\|^q_X.
\end{equation*}
Then there exists an equivalent norm $|\cdot|$ on $X$ such that:
\begin{equation}\label{eq:Pisier 1}
\begin{cases}
  \forall x,y\in X,\quad\|x\|\le|x|\le \mathfrak{m}_{q,X}\|x\|  \\
  \text{and},\quad |\frac{x+y}{2}|^q+\|\frac{x-y}{2}\|^q\le\frac{|x|^q+|y|^q}{2}.
\end{cases}
\end{equation}
In particular, $(X,|\cdot|)$ is uniformly convex of power type $q$ with constant $\delta=\mathfrak{m}_{q,X}^{-q}$.
\end{theorem}

\begin{remark}
The above result states that if $X$ has martingale cotype $q$, then it admits an equivalent norm that satisfies \eqref{eq:Pisier 1}. The converse to this is also true and it can be found in \cite[Remark 3.1]{Pisier} and \cite[Theorem 10.6 $(i)\implies(iii)$]{Pisier:2016}. In particular, if we assume that $X$ is uniformly convex of power type $q$ with constant $\delta$, then $X$ has martingale cotype $q$ with $\mathfrak{m}_{q,X}\le2\delta^{-\frac{1}{q}}$. As it can be seen below in our applications, we could work under the assumption of uniformly convex of power type $q$ spaces which would imply estimates in terms of the constant $\delta$.
\end{remark}

We will need two results in order to derive bounds for the family of operators $\{t\partial T_t\}_{t>0}$ on $Y=L_p(\Omega;X)$, where $1<p<\infty$, $X$ is a uniformly convex Banach space or a Banach space of martingale cotype $2\le q<\infty$. The first one is the following Rota's dilation theorem for positive Markovian operators, which we state following the formulation in \cite[Lemma 10]{X}:

\begin{lemma}[\cite{S}, Chapter IV and \cite{X}, Lemma 10]\label{lem:Rota}
Let $T=S^2$ with $S$ a  symmetric Markovian operator on $(\Omega,\mathcal{A}, \mu)$. Then there exist a larger measure space $(\widetilde{\Omega}, \widetilde{\mathcal{A}},\widetilde{\mu})$ containing $(\Omega,\mathcal{A},\mu)$, and a $\sigma$-subalgebra $\mathcal{B}$ of $\widetilde{\mathcal{A}}$ such that, for every $p\in[1,\infty)$ and every Banach space $X$,
\begin{equation}\label{eq:Rota}
  Tf=\Exp_\mathcal{A} \Exp_\mathcal{B} f,\quad\forall f\in L_p(\Omega,\mathcal{A},\mu;X),
\end{equation}
where $\Exp_{\mathcal{A}}$ is the conditional expectation relative to $\mathcal{A}$ (and similarly for $\Exp_\mathcal{B}$).
\end{lemma}

The fact that Rota's theorem remains valid for functions taking values in an arbitrary Banach space, as stated above, has been observed e.g. in \cite[after Theorem 2.5]{MTX}) and \cite[the beginning of the proof of Lemma 9]{X}. We indicate a proof for the reader's convenience.

\begin{proof}
For $X\in\{\R,\C\}$ (i.e., scalar-valued functions), a more general version of this lemma is formulated and proved in \cite[Chapter IV, Theorem 9]{S}, and restated in the form above as \cite[Lemma 10]{X}.

Let then $X$ be a general Banach space. If $f=\sum_{k=1}^K f_k\otimes x_k$, where $f_k\in L_p(\Omega,\mathcal{A},\mu)$ are scalar-functions and $x_k\in X$, the identity follows from the fact that it holds for each $f_k$ by linearity of both sides of \eqref{eq:Rota}.

For the general case, we note that functions of the form just discussed are dense in $L_p(\Omega,\mathcal{A},\mu;X)$. On the other hand, both symmetric Markovian operators (by Definition \ref{def:sds}) and conditional expectations (by their well-known basic properties) are positive linear operators (i.e., they map nonnegative functions to nonnegative functions). Such operators, initially defined on scalar-valued $L_p$ spaces, have canonical bounded linear extensions to the corresponding $L_p$ spaces of $X$-valued functions, for any Banach space $X$ (for this result, see e.g. \cite[Theorem 2.1.3]{HNVW}). Since \eqref{eq:Rota} holds for all $f$ in a dense subspace of $L_p(\Omega,\mathcal{A},\mu;X)$, and both sides depend continuously on $f\in L_p(\Omega,\mathcal{A},\mu;X)$, the identity remains valid for all $f$ as stated.
\end{proof}

The following elementary estimate improves the intermediate steps of the proof given in \cite[Lemma 9]{X}:

\begin{lemma}\label{lem:elem. ineq.}
Let $q\in(1,\infty)$ and $\delta\in(0,1]$. If $x\in[0,2]$ satisfies
\begin{equation*}
  x^q+\delta 2^q(x-1)_{+}^q\leq 2^q,\quad\text{where}\quad(x-1)_{+}=\max(x-1,0),
\end{equation*}
then $x\leq 2-\eps$, where $\displaystyle\eps:=\frac{2\delta}{1+2\delta}\frac{1}{q}\in(0,1)$.
\end{lemma}

\begin{proof}
Let us first observe the elementary inequalities
\begin{equation}\label{eq:basic}
  \forall t\in[0,1]:\quad (1-t)^\alpha\begin{cases}\leq 1-\alpha t, & \alpha\in(0,1), \\ \geq 1-\alpha t, & \alpha\in(1,\infty).\end{cases}
\end{equation}
Now, we proceed with the proof of our lemma.
Arguing by contradiction, we assume that $x>2-\eps$, thus $x-1> 1-\eps>0$, and hence
\begin{equation*}
  2^q\geq x^q+\delta 2^q(1-\eps)^q\geq x^q+\delta 2^q(1-q\eps)
\end{equation*}
by the assumption and \eqref{eq:basic} with $\alpha=q>1$. Thus
\begin{equation*}
  x^q\leq 2^q[1-\delta(1-q\eps)],
\end{equation*}
where $q\eps\in(0,1)$ and hence $\delta(1-q\eps)\in(0,1)$. By \eqref{eq:basic} with $\alpha=\frac{1}{q}\in(0,1)$, we have
\begin{equation*}
  x\leq 2[1-\delta(1-q\eps)]^{\frac{1}{q}}\leq 2\bigg[1-\frac{1}{q}\delta(1-q\eps)\bigg]=2-\frac{2\delta}{q}(1-q\eps).
\end{equation*}
Here
\begin{equation*}
  q\eps=\frac{2\delta}{1+2\delta},\quad 1-q\eps=\frac{1}{1+2\delta},\quad
  \frac{2\delta}{q}(1-q\eps)=\frac{2\delta}{1+2\delta}\frac{1}{q}=\eps.
\end{equation*}
So, assuming that $x>2-\eps$, we arrived at $x\leq 2-\eps$, a contradiction. Thus necessarily $x\leq 2-\eps$, as claimed.
\end{proof}

The following lemma is well known and immediate from the definition:

\begin{lemma}\label{lem:auxil.}
If $(X,\|\cdot\|)$ is uniformly convex of power type $q$ with $2\le q<\infty$ and constant $\delta>0$ appearing in \eqref{eq:Pisier}, then so is $Y=L_q(\Omega;(X,\|\cdot\|))$.
\end{lemma}

The following lemma is a quantitative elaboration of \cite[Lemma 9]{X}:

\begin{lemma}\label{lem:Xu}
Let $T=S^2$ with $S$ a symmetric Markovian operator on $(\Omega,\mathcal{A},\mu)$ and suppose that $X$ is uniformly convex of power type $q$ with $2\le q<\infty$ and $\delta>0$ appearing in \eqref{eq:Pisier}. Then
\begin{equation*}
  \|I-T\|\le 2-\eps,\quad\text{on}\quad Y=L_q(\Omega;X),
\end{equation*}
where $\eps:=\frac{2\delta}{1+2\delta}\frac{1}{q}\in(0,1)$.
\end{lemma}

\begin{proof}
Notice that by Lemma \ref{lem:auxil.}, $Y=L_q(\Omega;X)$ is also uniformly convex of power type $q$ with constant $\delta>0$ and $T$ is a contraction on $Y$. By Rota's dilation lemma (its $X$-valued version as formulated in Lemma \ref{lem:Rota}), we have
\begin{equation*}
  T=\Exp_\mathcal{A} \Exp_\mathcal{B}\big|_Y.
\end{equation*}
On the other hand, observe that $\Exp_{\mathcal A}$ acs as the identity on $Y$. Hence, for any $\lambda\in\C$, we have
\begin{equation*}
  \lambda + T
  =(\lambda\Exp_{\mathcal A}+\Exp_\mathcal{A} \Exp_\mathcal{B})\big|_Y
  =\Exp_{\mathcal A}(\lambda+\Exp_{\mathcal B})\big|_Y
  =:\Exp_{\mathcal A}(\lambda+P)\big|_Y,
\end{equation*}
where we abbreviated $P:=\Exp_{\mathcal B}$.

Let $y$ be a unit vector in $Y$. Using \eqref{eq:Pisier}, we get
\begin{equation*}
  \bigg\|\frac{\lambda y+Py}2\bigg\|^q+\delta\bigg\|\frac{\lambda y-Py}2\bigg\|^q\le \frac{1}{2}\,(|\lambda|^q+1),
\end{equation*}
where $\|\lambda y+Py\|\ge\|\Exp_\mathcal{A}(\lambda y+Py)\|\ge\|\lambda y+Ty\|$.
On the other hand (noting that $P$ and $\Exp_\mathcal{A}$ are contractive projections), we have
\begin{equation*}
   \big\|\lambda y-Py\big\|
   \geq \big\|P(\lambda y-Py)\big\|
   =\big\|\lambda P y-P^2y\big\|
   =\big\|\lambda P y-Py\big\|
   =\big\|(\lambda-1) P y\big\|,
\end{equation*}
and hence
\begin{equation*}
  \big\|\lambda y-Py\big\|\ge |1-\lambda| \,\big\|Py\big\|\ge |1-\lambda| \big(\big\|\lambda y+Py\big\|-|\lambda|\big)\ge  |1-\lambda| \big(\big\|\lambda y+Ty\big\|-|\lambda|\big).
\end{equation*}
Also clearly $\|\lambda y-Py\|\ge 0$, and hence in fact
\begin{equation*}
  \big\|\lambda y-Py\big\|\ge|1-\lambda|\big(\big\|\lambda y+Ty\big\|-|\lambda|\big)_{+},
\end{equation*}
where $\big(\big\|\lambda y+Ty\big\|-|\lambda|\big)_{+}=\max(\|\lambda y+Ty\|-|\lambda|,0)$. When $\|\lambda y+Ty\|$ approaches  $\|\lambda+T\|$, we then deduce
\begin{equation*}
  \bigg\|\frac{\lambda +T}2\bigg\|^q+\delta\, |1-\lambda|^q\bigg(\frac{\big\|\lambda +T\big\|-|\lambda|}2\bigg)_{+}^q\le \frac{1}{2}\,(|\lambda|^q+1).  
\end{equation*}
\noindent In particular, for $\lambda=-1$ we obtain
\begin{equation*}
  \|I-T\|^q+\delta\, 2^q(\|I-T\|-1)_{+}^q\le 2^q.
\end{equation*}
By Lemma \ref{lem:elem. ineq.} for $x:=\Norm{I-T}{}\in[0,2]$ this implies
\begin{equation*}
  \|I-T\|\le 2-\eps, 
\end{equation*}
where $\eps:=\frac{2\delta}{1+2\delta}\frac{1}{q}\in(0,1)$.
\end{proof}

Corollary \ref{coro:main} and Lemma \ref{lem:Xu} imply the following result, which is \cite[Remark 1.8(b)]{Pisier:1982} (see also \cite[Lemma 11]{X}):

\begin{corollary}\label{coro:main 1}
Suppose that $X$ is uniformly convex of power type $q$ with $2\le q<\infty$ and $\delta>0$ appearing in \eqref{eq:Pisier}. Moreover, suppose that $\{T_t\}_{t>0}$ is a symmetric diffusion semigroup on $(\Omega,\mathcal{A},\mu)$. Then the extension of $\{T_t\}_{t>0}$ to $Y=L_q(\Omega;X)$ is analytic. More precisely, we have 
\begin{equation*}
  \|T(z)\|\lesssim\frac{q}{\delta}\bigg(1+\log\bigg(\frac{q}{\delta}\bigg)\bigg),
\end{equation*}
for every $z\in\C$ such that $|\arg(z)|\lesssim\frac{\delta}{q}$ and $\{t\partial T_{t}\}_{t>0}$ is a uniformly bounded family of operators on $Y$, namely,
\begin{equation*}
  \sup_{t>0}\|t\partial T_{t}\|\lesssim\frac{q^2}{\delta^2}\bigg(1+\log\bigg(\frac{q}{\delta}\bigg)\bigg).
\end{equation*}
\end{corollary}

\begin{proof}
Applying Lemma \ref{lem:Xu} to $T=T_t$, we get
\begin{equation*}
  \|I-T_t\|\le 2-\eps,\quad\text{for all}\; t>0.
\end{equation*}
where $\eps=\frac{2\delta}{1+2\delta}\frac{1}{q}\approx\frac{\delta}{q}$.
Then by choosing $M=1$ and $\eps\approx\frac{\delta}{q}$ in Corollary \ref{coro:main} the proof follows.
\end{proof}

Now, Theorem \ref{thm:Pisier} and Corollary \ref{coro:main 1} imply the following estimate of $\sup_{t>0}\|t\partial T_{t}\|$ in terms of martingale cotype:

\begin{corollary}\label{coro:main 2}
Let $X$ be a Banach space of martingale cotype $2\le q<\infty$ with cotype constant $\mathfrak{m}_{q,X}$. Moreover, suppose that $\{T_t\}_{t>0}$ is a symmetric diffusion semigroup on $(\Omega,\mathcal{A},\mu)$. Then the extension of $\{T_t\}_{t>0}$ to $Y=L_q(\Omega;X)$ is analytic. More precisely, we have
\begin{equation}
  \|T(z)\|\lesssim q(\mathfrak{m}_{q,X})^{q+1}(1+\log(q)+q\log(\mathfrak{m}_{q,X})),
\end{equation}
for every $z\in\C$ such that $|\arg(z)|\lesssim\frac{1}{q\mathfrak{m}_{q,X}^q}$ and $\{t\partial T_{t}\}_{t>0}$ is a uniformly bounded family of operators on $Y$, namely,
\begin{equation}\label{main const.}
  \sup_{t>0}\|t\partial T_{t}\|\lesssim B,\quad B:=q^2(\mathfrak{m}_{q,X})^{2q+1}(1+\log(q)+q\log(\mathfrak{m}_{q,X})).
\end{equation}
\end{corollary}

\begin{proof}
Assuming that $|\cdot|$ is the equivalent norm on $X$ guaranteed by Theorem \ref{thm:Pisier} and $Y=L_q(\Omega;(X,|\cdot|))$ we can deduce by Theorem \ref{thm:Pisier} and Corollary \ref{coro:main 1} that 
\begin{equation*}
  |T(z)|\lesssim q\mathfrak{m}_{q,X}^q(1+\log(q)+q\log(\mathfrak{m}_{q,X}))   
\end{equation*}
and
\begin{equation*}
  \sup_{t>0}|t\partial T_{t}|\lesssim(q(\mathfrak{m}_{q,X})^q)^2(1+\log(q)+q\log(\mathfrak{m}_{q,X})).
\end{equation*}
Hence we conclude that
\begin{equation*}
\begin{split}
  \|T(z)\|&\le\mathfrak{m}_{q,X}|T(z)|  \\
  &\lesssim q(\mathfrak{m}_{q,X})^{q+1}(1+\log(q)+q\log(\mathfrak{m}_{q,X}))
\end{split}
\end{equation*}
and
\begin{equation*}
\begin{split}
  \sup_{t>0}\|t\partial T_{t}\|&\le\mathfrak{m}_{q,X}\sup_{t>0}|t\partial T_{t}|  \\
  &\lesssim q^2(\mathfrak{m}_{q,X})^{2q+1}(1+\log(q)+q\log(\mathfrak{m}_{q,X})).
\end{split}
\end{equation*}
\end{proof}

\section{Littlewood--Paley--Stein inequalities: First approach}\label{section:5}

In this section we will use the quantitative analyticity bounds of the previous section to obtain the following new quantitative version of \cite[Theorem 2 (i)]{X}:

\begin{theorem}\label{heat}
Let $X$ be a Banach space and $k$ a positive integer. If $X$ has martingale cotype $q$ with $2\le q<\infty$, then for every  symmetric diffusion semigroup $\{T_t\}_{t>0}$ and for every $1<p<\infty$ we have
\begin{equation}
  \bigg\|\left(\int_0^\infty\big\|t^k \partial^k T_t f\big\|_X^q\,\frac{dt}t\right)^{1/q}\bigg\|_{L_p(\Omega)}\lesssim_{p,q}  k^{k-1}B^{k+1}\mathfrak{m}_{q,X} \big\|f\big\|_{L_p(\Omega; X)},
\end{equation}
for all $f\in L_p(\Omega; X)$, where $B:=q^2(\mathfrak{m}_{q,X})^{2q+1}(1+\log(q)+q\log(\mathfrak{m}_{q,X}))$ is the constant appearing in \eqref{main const.} of Corollary \ref{coro:main 2}.

If $p=q$ and $k=1$, then we have the sharper bound
\begin{equation}\label{LPS eq:first approach}
  \bigg\|\left(\int_0^\infty\big\|t \partial T_t f\big\|_X^q\,\frac{dt}t\right)^{1/q}\bigg\|_{L_q(\Omega)}\lesssim B\mathfrak{m}_{q,X} \big\|f\big\|_{L_q(\Omega; X)},
\end{equation}
for all $f\in L_q(\Omega; X)$ and the implied constant is absolute.
\end{theorem}

\begin{remark}[see also \cite{X}, Remark 3]
Applied to the heat semigroup $\{H_t\}_{t>0}$ of $\R^n$, the above theorem implies a dimension free estimate for the $g$-function associated to $\{H_t\}_{t>0}$:
\begin{equation*}
  \bigg\|\bigg(\int_0^\infty\big\|t \partial H_t f\big\|_X^q\,\frac{dt}t\bigg)^{1/q}\bigg\|_{L_p(\R^n)}\lesssim_{p,q} B^2\mathfrak{m}_{q,X}\big\|f\big\|_{L_p(\R^n; X)},\forall f\in L_p(\R^n; X).  
\end{equation*}
when $X$ is of martingale cotype $q$. In addition, if $p=q$ the same theorem implies
\begin{equation*}
  \bigg\|\bigg(\int_0^\infty\big\|t \partial H_t f\big\|_X^q\,\frac{dt}t\bigg)^{1/q}\bigg\|_{L_q(\R^n)}\lesssim B\mathfrak{m}_{q,X}\big\|f\big\|_{L_q(\R^n; X)},\forall f\in L_q(\R^n; X).  
\end{equation*}
Compare this with the following estimate which was obtained in \cite[Theorem 17]{HN}:
\begin{equation*}
  \bigg(\int_0^\infty\big\|t \partial H_t f\big\|_{L_q(\R^n;Y)}^q\,\frac{dt}t\bigg)^{1/q}\lesssim\sqrt{n}\cdot\mathfrak{m}_{q,X}\|f\|_{L_q(\R^n;Y)},\forall f\in L_q(\R^n; Y) 
\end{equation*}
when $Y$ is a Banach space that admits an equivalent norm with modulus of uniform convexity of power type $q$. We note that the constant $B$ can be very large, so that the second bound is typically sharper in moderate dimensions $n$, but becomes inferior as $n\rightarrow\infty$.
\end{remark}

In order to prove Theorem \ref{heat} the following lemma, due to Naor and one of us \cite[Lemma 24]{HN} will play a crucial role in our argument:

\begin{lemma}[\cite{HN}, Lemma 24]\label{lem:HN} Fix $q\in [2,\infty)$, $\alpha\in (1,\infty)$ and a Banach space $(X,\|\cdot\|_X)$ of martingale cotype $q$ with constant $\mathfrak{m}_{q,X}$. Suppose that $\{T_t\}_{t>0}$ is a symmetric diffusion semigroup on a measure space $(\Omega,X)$. Then for any $f\in L_q(\Omega;X)$ we have
\begin{equation*}
  \bigg(\int_0^\infty\Norm{(T_t-T_{\alpha t})f}{L_q(\Omega;X)}^q\frac{\ud t}{t}\bigg)^{\frac{1}{q}}\leq(\log(\alpha))^{\frac{1}{q}}\mathfrak{m}_{q,X}\Norm{f}{L_q(\Omega;X)}.
\end{equation*}
\end{lemma}

The following lemma shows Theorem \ref{heat} in the case $p=q$ and $k=1,2$, but the bound provided by the lemma is worse than that asserted in the theorem when $k>2$. Thus, some more work will be needed further below to obtain Theorem \ref{heat} as stated.

\begin{lemma}[quantitative version of \cite{X}, Lemma 13]\label{p=q}
Let $X$ be a Banach space of martingale cotype $2\le q<\infty$ with cotype constant $\mathfrak{m}_{q,X}$ and $k$ be a positive integer. Moreover, suppose that $\{T_t\}_{t>0}$ is a symmetric diffusion semigroup on a measure space $(\Omega,X)$. Then for any $f\in L_q(\Omega;X)$ we have
\begin{equation}\label{pqk}
  \bigg(\int_0^\infty\big\|t^k\partial^k T_tf\big\|_{L_q(\Omega;X)}^q\,\frac{dt}{t}\bigg)^{1/q}\lesssim k^k B^k\mathfrak{m}_{q,X}\|f\|_{L_q(\Omega;X)},
\end{equation}
where $B:=q^2(\mathfrak{m}_{q,X})^{2q+1}(1+\log(q)+q\log(\mathfrak{m}_{q,X}))$ is the constant appearing in \eqref{main const.} of Corollary \ref{coro:main 2}.
\end{lemma}

\begin{remark}
Later on we will only use this lemma in the case $k=1$ and we will get a {different} bound {form} the one appearing in \eqref{pqk} for $k>1$. 
\end{remark}

\begin{proof}
We will use the idea of the proof of \cite[Theorem 17]{HN}.
By the semigroup identity $\partial T_{t+s}=\partial T_t\,T_s$ and the convergence
\begin{equation*}
  \|\partial T_t\|=\frac{1}{t}\|t\partial T_t\|\leq\frac{B}{t}\underset{t\to\infty}{\longrightarrow} 0
\end{equation*}
from \eqref{main const.}, we can expand
\begin{equation}
  \partial T_tf=\sum_{k=-1}^\infty\big(\partial T_{2^{k+1}t}- \partial T_{2^{k+2}t}\big)f=\sum_{k=-1}^\infty\partial T_{2^{k}t}\big(T_{2^{k}t}- T_{3\cdot2^{k}t}\big)f.
\end{equation}
Then by the triangle inequality we can estimate
\begin{equation}\label{main estim.1}
\begin{split}
  \big(\int_0^\infty &\big\|t\partial T_tf\big\|_{L_q(\Omega;X)}^q\,\frac{dt}{t}\big)^{1/q} \\
  &\le \sum_{k=-1}^\infty \bigg(\int_0^\infty\big\|t\partial T_{2^{k}t}\big(T_{2^{k}t}- T_{3\cdot2^{k}t}\big)f\big\|_{L_q(\Omega;X)}^q\,\frac{dt}{t}\bigg)^{1/q}  \\
  &= \sum_{k=-1}^\infty2^{-k}\bigg(\int_0^\infty \big\|t\partial T_t\big(T_{t}-T_{3t}\big)f\big\|_{L_q(\Omega;X)}^q\,\frac{dt}{t}\bigg)^{1/q}  \\
  &=4\bigg(\int_0^\infty\big\|t\partial T_t\big(T_{t}- T_{3t}\big)f\big\|_{L_q(\Omega;X)}^q\,\frac{dt}{t}\bigg)^{1/q}.
\end{split}
\end{equation}

We are now in a position of using Corollary \ref{coro:main 2} as follows:
\begin{equation}\label{eq1:main}
\big\|t\partial T_t\big(T_{t}- T_{3t}\big)f\big\|_{L_q(\Omega;X)}\lesssim B \big\|\big(T_{t}- T_{3t}\big)f\big\|_{L_q(\Omega;X)}\,,\quad\forall\,t>0.
\end{equation}
where $B:=q^2(\mathfrak{m}_{q,X})^{2q+1}(1+\log(q)+q\log(\mathfrak{m}_{q,X}))$.

Combining the above inequalities together with Lemma \ref{lem:HN}, we derive
\begin{equation*}
\begin{split}
  \bigg(\int_0^\infty\big\|t\partial  T_tf\big\|_{L_q(\Omega;X)}^q\,\frac{dt}{t}\bigg)^{1/q}&\lesssim 
  B\bigg(\int_0^\infty\big\|\big(T_{t}- T_{3t}\big)f\big\|_{L_q(\Omega;X)}^q\,\frac{dt}{t}\bigg)^{1/q}  \\
  &\lesssim B(\log(3))^{\frac{1}{q}}\mathfrak{m}_{q,X}\big\|f\big\|_{L_q(\Omega;X)}  \\
  &\lesssim B\mathfrak{m}_{q,X}\big\|f\big\|_{L_q(\Omega;X)}.
\end{split}
\end{equation*}
This is \eqref{pqk} for $k=1$. To handle a general $k$, by the semigroup identity $T_{t+s}=T_tT_s$, we have
\begin{equation*} 
  t^k\partial^k T_t=k^k\left(\frac{t}k\,\partial T_{\frac{t}k}\right)^k.
\end{equation*}
This implies that 
\begin{equation}\label{eq2:main}
  \bigg(\int_0^\infty\big\|t^k\partial^k T_tf\big\|_{L_q(\Omega;X)}^q\,\frac{dt}{t}\bigg)^{1/q}=k^{k}\bigg(\int_0^\infty\big\|(t\partial T_{t})^{k} f\big\|_{L_q(\Omega;X)}^{q}\,\frac{dt}{t}\bigg)^{1/q}.
\end{equation}
Thus, by \eqref{eq1:main}, \eqref{eq2:main} and the already proved inequality, we obtain
\begin{equation*}
\begin{split}
  \eqref{eq2:main}&\lesssim k^{k}\bigg(\int_0^\infty\big\|(t\partial T_{t})^{k-1}(t\partial T_{t}) f\big\|_{L_q(\Omega;X)}^{q}\,\frac{dt}{t}\bigg)^{1/q}  \\
  &\lesssim k^k B^{k-1}\bigg(\int_0^\infty\big\|t\partial  T_tf\big\|_{L_q(\Omega;X)}^q\,\frac{dt}{t}\bigg)^{1/q}  \\
  &\lesssim k^k B^k\mathfrak{m}_{q,X}\|f\|_{L_q(\Omega;X)}.
\end{split}
\end{equation*}
The lemma is thus proved.
\end{proof}

To show Theorem \ref{heat} for any $1<p<\infty$, we will use Stein's complex interpolation machinery. To that end, we will need the notion of fractional integrals. For a (nice) function $\varphi$ on $(0,\,\infty)$ we define
\begin{equation*}
  \mathrm I^\alpha\varphi(t)=\frac1{\Gamma(\alpha)}\,\int_0^t (t-s)^{\alpha-1}\varphi(s)ds,\quad t>0.
\end{equation*}
The integral on the right hand side is well defined for any $\alpha\in\C$ with $\Re\,\alpha>0$; moreover, $\mathrm I^\alpha\varphi$ is analytic in the right  half complex plane $\Re\,\alpha>0$. Using integration by parts, Stein showed in \cite[Section III.3]{S} that $\mathrm I^\alpha\varphi$ has an analytic continuation to the whole complex plane, which satisfies the following properties:
\begin{enumerate}[(1)]
  \item $\mathrm I^\alpha\mathrm I^\beta\varphi=\mathrm I^{\alpha+\beta}\varphi$ for any $\alpha,\beta\in\C$;
  \item $\mathrm I^0\varphi=\varphi$; and 
  \item $\mathrm I^{-k}\varphi=\partial^k\varphi$ for any positive integer $k$.
\end{enumerate}
We will apply $\mathrm I^\alpha$ to $\varphi$ defined by $\varphi(s)=T_sf$ for a given function $f$ in $L_p(\Omega; X)$ and set
\begin{equation*}
  \mathrm M^\alpha_tf=t^{-\alpha} \mathrm I^\alpha \varphi(t)\;\text{with }\; \varphi(s)=T_sf.
\end{equation*} 
Note that
\begin{equation*}
  \mathrm M^1_tf=\frac{1}{t}\,\int_0^tT_sfds,\;\; \mathrm M^0_tf=T_tf\;\text{ and}\; \mathrm M^{-k}_tf=t^k \partial^k T_tf\;\text{for}\;k\in\N.
\end{equation*}
The following lemma is \cite[Theorem 2.3]{MTX}:

\begin{lemma}[\cite{MTX}, Theorem 2.3 and \cite{X}, Lemma 14]\label{average}
Let $q$ and $X$ be as in Theorem \ref{heat}. Then for any $1<p<\infty$ we have
\begin{equation*}
  \bigg\|\bigg(\int_0^\infty\big\|t\partial \mathrm M^1_tf\big\|_{X}^q\,\frac{dt}{t}\bigg)^{1/q}\bigg\|_{L_p(\Omega)}\lesssim_{p,q}\mathfrak{m}_{q,X}\big\|f\big\|_{L_p(\Omega;X)}\,,\quad\forall\,f\in L_p(\Omega; X).  
\end{equation*}
\end{lemma}

\begin{lemma}[\cite{X}, Lemma 15]\label{a-b}
Let $\alpha$ and $\beta$ be complex numbers such that $\Re\alpha>\Re\beta>-1$. Then for any positive integer $k$
\begin{equation*}
  \bigg(\int_0^\infty\big\|t^k\partial^k \mathrm M^\alpha_tf\big\|_{X}^q\,\frac{dt}{t}\bigg)^{\frac{1}{q}}
  \lesssim_{\Re\alpha,\Re\beta}e^{\pi
  |{\rm Im}\, (\alpha-\beta)|}\left(\int_0^\infty\big\|t^k\partial^k \mathrm M^\beta_tf\big\|_{X}^q\,\frac{dt}{t}\right)^{\frac{1}{q}}
\end{equation*}
on $\Omega$.
\end{lemma}

Combining Lemma \ref{average} and Lemma \ref{a-b} with $k=\beta=1$, we get:
\begin{lemma}[\cite{X}, Lemma 16]\label{a>1}
For any $1<p<\infty$ and $\alpha\in\C$ with $\Re\alpha>1$
\begin{equation*}
  \bigg\|\bigg(\int_0^\infty\big\|t\partial \mathrm M^\alpha_tf\big\|_{X}^q\,\frac{dt}{t}\bigg)^{1/q}\bigg\|_{L_p(\Omega)}
  \lesssim_{p,q,\Re\alpha}\mathfrak{m}_{q,X} e^{\pi|{\rm Im}\,\alpha|}\big\|f\big\|_{L_p(\Omega;X)},
\end{equation*}
for all $f\in L_p(\Omega; X)$.
\end{lemma}

Now we provide the following new quantitative version of \cite[Lemma 17]{X} which we will use in the proof of Theorem \ref{heat}:

\begin{lemma}[\cite{X}, Lemma 17]\label{p=q2}
For any $\alpha\in\C$
\begin{equation}\label{q-a}
\begin{split}
  \bigg\|\bigg(\int_0^\infty &\big\|t\partial \mathrm M^\alpha_tf\big\|_{X}^q\,\frac{dt}{t}\bigg)^{1/q}\bigg\|_{L_q(\Omega)} \\
  &\lesssim_{\Re\alpha}(|\alpha|+{1})^{N(\Re{\alpha})} B^{1+N(\Re\alpha)}\mathfrak{m}_{q,X} e^{\pi |{\rm Im}\,\alpha|}  \big\|f\big\|_{L_q(\Omega;X)}, 
\end{split}  
\end{equation}
for all $f\in L_q(\Omega; X)$, where
\begin{equation}\label{lem:function N}
N(\Re{\alpha}):=\begin{cases}
      0, \text{if $\Re{\alpha}>0$},  \\
      \text{smallest $n\ge 0$ such that $n>-\Re{\alpha}$},
\end{cases}
\end{equation}
and $B:=q^2(\mathfrak{m}_{q,X})^{2q+1}(1+\log(q)+q\log(\mathfrak{m}_{q,X}))$ is the constant appearing in \eqref{main const.} of Corollary \ref{coro:main 2}.
\end{lemma}

\begin{proof}
Combining Lemmas \ref{p=q} and \ref{a-b} with $\beta=0$, we deduce that for a positive integer $k$ and $\alpha\in\C$ with $\Re\alpha>0$
\begin{equation}\label{iterate}
  \bigg\|\bigg(\int_0^\infty\big\|t^k\partial^k \mathrm M^\alpha_tf\big\|_{X}^q\,\frac{dt}{t}\bigg)^{1/q}\bigg\|_{L_q(\Omega)}
  \lesssim_{\Re\alpha} k^{k}B^k\mathfrak{m}_{q,X} e^{\pi |{\rm Im}\,\alpha|}\big\|f\big\|_{L_q(\Omega;X)},    
\end{equation}
for all $f\in L_q(\Omega; X)$.

In order to prove the general case for any $\alpha\in\C$, we will use an induction argument. In particular, we will prove the following estimate for any $\alpha\in\C$ with $\Re\alpha>-n$, $n\in\N$
\begin{equation}\label{iterate 1}
  \|t^{k}\partial^{k}\mathrm M^{\alpha}_t f\|_{L_q(\Omega; L_q((\R_+,\frac{dt}{t});X))}\le\phi_n(k,\alpha)\big\|f\big\|_{L_q(\Omega;X)},
\end{equation}
where
\begin{equation}\label{ineq:phi}
  \phi_n(k,\alpha)\le
  {k^k B^k C(\Re\alpha+n)\mathfrak{m}_{q,X}e^{\pi |{\rm Im}\,\alpha|}\psi_n(k,\alpha),  }
\end{equation}
$C(\Re\alpha+n)$ is a positive constant that depends on $\Re\alpha+n$ and
\begin{equation}\label{ineq:psi}
  \psi_n(k,\alpha)\le
  {  c^n(k+|\alpha|)^n B^n.}
\end{equation}

Noting that for any $\alpha\in\C$
\begin{equation*}
  \partial\mathrm M^\alpha_t=-\alpha t^{-1} \mathrm M^\alpha_t+ t^{-1} \mathrm M^{\alpha-1}_t,
\end{equation*}
we have
\begin{equation}\label{a-1}
  t^k\partial^k \mathrm M_t^{\alpha-1}
  =(k+\alpha) t^k\partial^k \mathrm M^{\alpha}_t+t^{k+1}\partial^{k+1} \mathrm M^{\alpha}_t.    
\end{equation}

Observe that for $n=0$, by what is already proved in $\eqref{iterate}$ we have 
\begin{equation*}
  \|t^{k}\partial^{k}\mathrm M^{\alpha}_t f\|_{L_q(\Omega; L_q((\R_+,\frac{dt}{t});X))}\le\phi_0(k,\alpha)\big\|f\big\|_{L_q(\Omega;X)},
\end{equation*}
where $\phi_0(k,\alpha)\le k^{k}B^k C(\Re\alpha)\mathfrak{m}_{q,X} e^{\pi |{\rm Im}\,\alpha|}\psi_0(k,\alpha)$ and $\psi_0(k,\alpha)=1$.

Now, let us assume that \eqref{iterate 1} holds for some $n\in\N$ and we prove it for $n+1$. To be more precise, if $\Re\tilde{\alpha}>-n$, where $\tilde{\alpha}:=\alpha+1$, then using the induction hypothesis \eqref{iterate 1} and \eqref{a-1} we deduce that
\begin{multline}\label{iteration 2}
  \|t^{k}\partial^{k}\mathrm M^{\tilde{\alpha}-1}_t f\|_{L_q(\Omega; L_q((\R_+,\frac{dt}{t});X))}  \\
  \le|k+\tilde{\alpha}|\|t^{k}\partial^{k}\mathrm M^{\tilde{\alpha}}_t f\|_{L_q(\Omega; L_q((\R_+,\frac{dt}{t});X))}+\|t^{k+1}\partial^{k+1}\mathrm M^{\tilde{\alpha}}_t f\|_{L_q(\Omega; L_q((\R_+,\frac{dt}{t});X))}  \\
  \le(|k+\tilde{\alpha}|\phi_n(k,\tilde{\alpha})+\phi_n(k+1,\tilde{\alpha}))\big\|f\big\|_{L_q(\Omega;X)}  \\
  =(|k+\alpha+1|\phi_n(k,\alpha+1)+\phi_n(k+1,\alpha+1))\big\|f\big\|_{L_q(\Omega;X)}.
\end{multline}

Hence, {
\begin{equation}\label{iteration 3}
\begin{split}
&\phi _{n+1}(k,\alpha )= k^k B^kC(\text{Re}\alpha+n+1)\mathfrak{m}_{q,X}e^{\pi|\text{Im}\alpha|}\psi_{n+1}(k,\alpha) \\&\le |k+\alpha +1|\phi _n(k,\alpha +1)+\phi _n(k+1,\alpha +1) \\ &\le |k+\alpha+1| k^k B^k C(\text{Re}\alpha+1+n)\mathfrak{m}_{q,X}e^{\pi|\text{Im}\alpha|}\psi_n(k,\alpha+1) \\ &\qquad+(k+1)^{k+1}B^{k+1}C(\text{Re}\alpha+1+n)\mathfrak{m}_{q,X}e^{\pi|\text{Im}\alpha|}\psi_n(k+1,\alpha+1).
\end{split}
\end{equation}}
and for {$\psi_{n+1}(k,\alpha)$} we have {
\begin{equation*}
    \psi_{n+1}(k,\alpha) \le |k+\alpha +1|\psi_n(k,\alpha+1)  +\Big(\frac{k+1}{k}\Big)^k(k+1)B\psi_n(k+1,\alpha+1),
\end{equation*}}
where {$(\frac{k+1}{k})^k\leq e$}.

Now, by the induction hypothesis {\eqref{ineq:psi}} for the functions {$\psi_n(k,\alpha)$} and \eqref{iteration 3} we have {
\begin{equation}\label{iteration 4}
\begin{split}
\psi _{n+1}(k,\alpha )&\le |k+\alpha +1|c^n(k+|\alpha+1|)^n B^n \\
&\qquad+e(k+1)B\cdot c^n(k+1+|\alpha+1|)^n B^n \\&\le c^n B^{n+1}(k+|\alpha|+2)^{n+1}(1+e).
\end{split}
\end{equation}
Here
\begin{equation}\label{iteration 5}
\begin{split}
  (k+|\alpha|+2)^{n+1}(1+e) 
  &=(k+|\alpha|)^{n+1}\Big(1+\frac{2}{k+|\alpha|}\Big)^{n+1}(1+e) \\
  &\le (k+|\alpha|)^{n+1}\Big(1+\frac{2}{\max(1,n)}\Big)^{n+1}(1+e) \\
  &\le(k+|\alpha|)^{n+1}c,
\end{split}
\end{equation}
since for the minimal nonnegative $n>-\Re\alpha$ we have
\begin{equation*}
  k+|\alpha|\geq 1+|\alpha|\geq \max(1,n).
\end{equation*}
Now, by the estimate \eqref{iteration 4} for the functions $\psi_{n+1}(k,\alpha)$ and \eqref{iteration 5} we can conclude
\begin{equation}\label{iteration 6}
\begin{split}
  \psi _{n+1}(k,\alpha )
  &\le c^nB^{n+1}(k+|\alpha|+2)^{n+1}(1+e)  \\
  &\leq c^n B^{n+1}(k+|\alpha|)^{n+1}c = c^{n+1}B^{n+1} (k+|\alpha|)^{n+1}.
\end{split}
\end{equation}
}
This completes the proof of \eqref{ineq:psi} by induction. Therefore, by combining {\eqref{iterate 1}, \eqref{ineq:phi} and \eqref{ineq:psi}}
 we obtain the following estimate for any $\alpha\in\C$ with $\Re\alpha>-n$ {
\begin{equation*}
\begin{split}
  &  \Vert t^{k}\partial ^{k}\textrm{M}^{\alpha }_t f\Vert _{L_q(\Omega ; L_q(({\mathbb {R}}_+,\frac{dt}{t});X))}\le \phi _n(k,\alpha )\big \Vert f\big \Vert _{L_q(\Omega ;X)} \\{}
   &  \quad \le k^kB^kC({\text {Re}}\alpha +n){\mathfrak {m}}_{q,X}e^{\pi |\textrm{Im}\,\alpha |}c^n (k+|\alpha |)^n B^n\big \Vert f\big \Vert _{L_q(\Omega ;X)} \\
    &  \quad \lesssim _{{\text {Re}}\alpha }k^k B^{k+N(\text{Re}\alpha)}(k+|\alpha |)^{N({\text {Re}}{\alpha })} {\mathfrak {m}}_{q,X} e^{\pi |\textrm{Im}\,\alpha |}\big \Vert f\big \Vert _{L_q(\Omega ;X)}, 
\end{split}
\end{equation*}}
where {$n=N(\Re{\alpha})$} is as in \eqref{lem:function N}. In particular for $k=1$, we have \eqref{q-a}.
\end{proof}

Now we are ready to show Theorem \ref{heat}:

\begin{proof}[Proof of Theorem \ref{heat}]
We will prove by an induction argument the following more general statement: for any $\alpha\in\C$
\begin{multline}\label{heat a}
  \bigg\|\bigg(\int_0^\infty\big\|t^k\partial^k \mathrm M^\alpha_tf\big\|_{X}^q\,\frac{dt}{t}\bigg)^{1/q}\bigg\|_{L_p(\Omega)}  \\
  \le C(k+|\alpha|)^{k-1}B^{k+1+\abs{\Re\alpha}} \mathfrak{m}_{q,X} \big\|f\big\|_{L_p(\Omega;X)},
\end{multline}
for all $f\in L_p(\Omega; X)$, where $C=C(p,q,\alpha)$ is a positive constant that depends on $p,q$ and $\alpha$, and  $B:=q^2(\mathfrak{m}_{q,X})^{2q+1}(1+\log(q)+q\log(\mathfrak{m}_{q,X}))$.

We fix $\alpha\in\C$ and choose $\theta\in(0,1),\; r\in(1,\infty)$, $\alpha_0, \alpha_1\in\C$ such that the following hold
\begin{equation}\label{a-2}
\begin{split}
  \frac{1}{p}=\frac{1-\theta}{q} + \frac{\theta}{r}\,,\;
  \alpha=(1-\theta)\alpha_0 +\theta\,\alpha_1,\;  \\
  \Re\alpha_1>1\;\text{and }\;\Im\alpha_0=\Im\alpha_1=\Im\alpha.      
\end{split}
\end{equation}
Then the classical complex interpolation of vector-valued $L_p$-spaces (see \cite{BL}) implies
\begin{equation*}
  L_p(\Omega; X)=\big[L_q(\Omega; X),\, L_r(\Omega; X)\big]_\theta.
\end{equation*}
Hence for any $f\in L_p(\Omega;X)$ with $\|f\|_{L_p(\Omega;X)}<1$ there exists a continuous function $F$ from the closed strip $\{z\in\C: 0\le\Re z\le1\}$ to $L_q(\Omega; X)+L_r(\Omega; X)$, which is analytic in the interior and satisfies
\begin{equation*}
  F(\theta)=f,\quad \sup_{y\in\R}\big\|F({\rm i}y)\big\|_{L_q(\Omega; X)}<1\;\text{and}\; \sup_{y\in\R}\big\|F(1+{\rm i}y)\big\|_{L_r(\Omega; X)}<1.  
\end{equation*}
Define
\begin{equation*}
  \mathcal F_t(z)=e^{z^2-\theta^2}\,t\partial \mathrm M^{(1-z)\alpha_0+z\alpha_1}_tF(z).  
\end{equation*}
Viewed as a function of $z$ on the strip $\{z\in\C: 0\le\Re z\le1\}$, the function $\mathcal F$ takes values in $L_q(\Omega; L_q(\R_+; X))+L_r(\Omega; L_q(\R_+; X))$, where $\R_+$ is equipped with the measure $\frac{dt}{t}$. By the analyticity of $\mathrm M^{(1-z)\alpha_0+z\alpha_1}$ in $z$, we see that $\mathcal F$ is analytic in the interior of the strip. Moreover,
by Lemma \ref{p=q2} applied to $(1-iy)\alpha_0+iy\alpha_1$, where real part is $\Re\alpha_0$ by \eqref{a-2}, in place of $\alpha$ we get
\begin{equation*}
\begin{split}
  &\bigg\|\bigg(\int_0^\infty\big\|\mathcal F_t({\rm i}y)\big\|_{X}^q\,\frac{dt}{t}\bigg)^{1/q}\bigg\|_{L_q(\Omega)}  \\
  &\lesssim_{\Re\alpha_0}(|(1-iy)\alpha_0+iy\alpha_1|+2)^{N(\Re{\alpha_0})}B^{1+N(\Re\alpha_0)}\mathfrak{m}_{q,X} \\
  &\qquad\times e^{-y^2-\theta^2}e^{\pi (|{\rm Im}\,\alpha|+|{\rm Re}(\alpha_1-\alpha_0)y|)}  \\
  & \lesssim_{\Re\alpha_0}G(y)B^{1+N(\Re\alpha_0)}\mathfrak{m}_{q,X},
\end{split}
\end{equation*}
for all $y\in\R$ and
\begin{equation*} 
  G(y):=\bigg(|\alpha_0|\sqrt{1+y^2}+|y||\alpha_1|+2\bigg)^{1+|\Re\alpha_0|}e^{-y^2}e^{|{\rm Re}(\alpha_1-\alpha_0)y|)}e^{\pi |{\rm Im}\,\alpha|}.  
\end{equation*}
Notice that $y\mapsto G(y)$ is continuous on $\R$ and $\lim_{y\rightarrow\pm\infty}G(y)=0$. Therefore it attains some maximum $M_0(\alpha_0,\alpha_1,\alpha)=M_1(p,q,\alpha)$, where the last equality holds since the choice of $\alpha_0$ and $\alpha_1$ in \eqref{a-2} depends only on $p,q$ and $\alpha$.
Hence
\begin{equation*}
  \sup_{y\in\R}\|\mathcal F({\rm i}y)\|_{L_q(\Omega; L_q((\R_+,\frac{dt}{t}); X))}\lesssim_{p,q,\alpha}B^{1+N(\Re\alpha_0)}\mathfrak{m}_{q,X}.
\end{equation*}
Similarly, Lemma \ref{a>1}, with $-iy\alpha_0+(1+iy)\alpha_1$ in place of $\alpha$, implies
\begin{equation*}
  \sup_{y\in\R}\|\mathcal F(1+{\rm i}y)\|_{L_r(\Omega; L_q((\R_+,\frac{dt}t); X))}\lesssim_{p,q,\alpha}\mathfrak{m}_{q,X}.
\end{equation*}

We then deduce that $\mathcal F(\theta)$ belongs to the complex interpolation space
\begin{equation}\label{interp-sp}
  \big[L_q(\Omega; L_q(\R_+; X)),\, L_r(\Omega; L_q(\R_+; X))\big]_{\theta}
\end{equation}
with norm majorized by  {
\begin{equation*}
 \lesssim_{p,q,\alpha}(B^{1+N(\Re\alpha_0)}\mathfrak{m}_{q,X})^{1-\theta}(\mathfrak{m}_{q,X})^\theta
  = B^{(1-\theta)(1+N(\Re\alpha_0))}\mathfrak{m}_{q,X},
\end{equation*}
where
\begin{equation*}
\begin{split}
  (1-\theta)(1+N(\Re\alpha_0))
  &=(1-\theta)\Big\{1+N\Big(\frac{1}{1-\theta}\Re\alpha{-}\frac{\theta}{1-\theta}\Re\alpha_1\Big)\Big\} \\
  &\overset{(*)}{=}(1-\theta)\Big\{1+N(\frac{1}{1-\theta}\Re\alpha{-}\frac{\theta}{1-\theta}\Big)\Big\} \\
  &\leq(1-\theta)\Big\{2+\Big(\frac{1}{1-\theta}\abs{\Re\alpha}+\frac{\theta}{1-\theta}\Bigg)\Big\} \\
  &2-2\theta+\abs{\Re\alpha}+\theta\leq 2+\abs{\Re\alpha},
\end{split}
\end{equation*}
since we can choose $\Re\alpha_1$ just slightly larger than $1$ to ensure the equality in $(*)$.}
However, the latter space coincides with  $L_p(\Omega;L_q(\R_+;X))$ isometrically. Since
\begin{equation*}
  \mathcal F_t(\theta)=t\partial \mathrm M^{\alpha}_tF(\theta)=t\partial \mathrm M^{\alpha}_tf,  
\end{equation*}
we get 
\begin{equation*}
  \bigg\|\bigg(\int_0^\infty\big\|t\partial \mathrm
  M^\alpha_tf\big\|_{X}^q\,\frac{dt}{t}\bigg)^{1/q}\bigg\|_{L_p(\Omega)}
  \lesssim_{p,q,\alpha}B^{2+\abs{\Re\alpha}}\mathfrak{m}_{q,X}\big\|f\big\|_{L_p(\Omega;X)},
\end{equation*}
for all $f\in L_p(\Omega; X)$. This is \eqref{heat a} for $k=1$. 

Now, let us assume that \eqref{heat a} holds for some $k$. Then using \eqref{a-1} and the induction hypothesis \eqref{heat a}, we get
\begin{equation*}
\begin{split}
  \| &t^{k+1} \partial^{k+1}\mathrm M^\alpha_t f\|_{L_p(\Omega; L_q((\R_+,\frac{dt}{t});X))}  \\ 
  &\le \|t^{k}\partial^{k}\mathrm M^{\alpha-1}_t f\|_{L_p(\Omega; L_q((\R_+,\frac{dt}{t});X))}+|k+\alpha|\|t^{k}\partial^{k}\mathrm M^\alpha_t f\|_{L_p(\Omega; L_q((\R_+,\frac{dt}{t});X))}  \\
 &\le {  C\bigg((k+|\alpha-1|)^{k-1}B^{k+1+\abs{\Re\alpha-1}} } \\
  &\qquad { +\abs{k+\alpha}(k+|\alpha|)^{k-1}B^{k+1+\abs{\Re\alpha}}\bigg)\mathfrak{m}_{q,X} \big\|f\big\|_{L_p(\Omega;X)} } \\
  & { =C(k+1+|\alpha|)^{k-1}(1+k+|\alpha|)B^{k+2+\abs{\Re\alpha}}\mathfrak{m}_{q,X} \big\|f\big\|_{L_p(\Omega;X)} } \\
  & { =C((k+1)+|\alpha|)^{(k+1)-1}B^{(k+1)+1+\abs{\Re\alpha}}\mathfrak{m}_{q,X} \big\|f\big\|_{L_p(\Omega;X)}. }
\end{split}
\end{equation*}
Thus, we derive \eqref{heat a} for any $k$. Theorem \ref{heat} corresponds to \eqref{heat a} for $\alpha=0$. In particular, we get 
\begin{equation*}
\begin{split}
  \bigg\|\left(\int_0^\infty\big\|t^k \partial^k T_t f\big\|_X^q\,\frac{dt}t\right)^{1/q}\bigg\|_{L_p(\Omega)}&\leq C(p,q)k^{k-1}B^{k+1}\mathfrak{m}_{q,X}\big\|f\big\|_{L_p(\Omega;X)},
\end{split}
\end{equation*}
for all $f\in L_p(\Omega; X)$. Thus, the theorem is completely proved.
\end{proof}

\section{Littlewood--Paley--Stein inequalities: Second approach}\label{section:6}

In this section we obtain an alternative estimate for the constant appearing in \eqref{LPS eq:first approach} of Theorem \ref{heat} by combining our Corollary \ref{coro:main 2} with another result of Xu \cite[Theorem 1.4]{Xu:2021}. Before we proceed we need to recall some notions from \cite{Xu:2021}.

Recall that an operator $T$ on $L_p(\Omega)$ ($1\le p\le\infty$) is {\em regular} (more precisely, {\em contractively regular}) if
\begin{equation*}
  \|\sup_{k}|T(f_k)|\|_p\le\|\sup_{k}|f_k|\|_p
\end{equation*}
for all finite sequences $\{f_k\}_{k\ge 1}$ in $L_p(\Omega)$.

Recall that $\{T_t\}_{t>0}$ is analytic on $L_p(\Omega;X)$ if $\{T_t\}_{t>0}$ extends to a bounded analytic function from an open sector $\Sigma_{\beta_0}=\{z\in\C:|\arg(z)|<\beta_0\}$ to $B(L_p(\Omega;X))$ for some $0<\beta_0\le\frac{\pi}{2}$, where $B(Y)$ denotes the space of bounded linear operators on a Banach space $Y$. In this case,
\begin{equation}\label{analyticity}
  {\bf T}_{\beta_0}=\sup\{\|T_{z}\|_{B(L_p(\Omega;X))}: z\in\Sigma_{\beta_0}\}<\infty.
\end{equation}

Now, we can state the following result of Xu:

\begin{theorem}[\cite{Xu:2021}, Theorem 1.4]
Let $X$ be a Banach space, $2\le q<\infty$ and $\{T_t\}_{t>0}$ be a strongly continuous semigroup of regular operators on $L_p(\Omega)$ for a single $1<p<\infty$. Assume additionally that $\{T_t\}_{t>0}$ satisfies \eqref{analyticity}. Let
$\beta_q=\beta_0\min(\frac{p}{q},\frac{p'}{q'})$. If $X$ has martingale cotype q, then 
\begin{equation}\label{Xu's estimate}
\begin{split}
  \bigg\|\left(\int_0^\infty\big\|t \partial T_t f\big\|_X^q\,\frac{dt}t\right)^{1/q}\bigg\|_{L_p(\Omega)}&\lesssim
  \beta_q^{-3}{\bf T}_{\beta_0}^{\min(\frac{p}{q},\frac{p'}{q'})}\max(p^{\frac{2}{q}},p'^{1+\frac{1}{q'}})\mathfrak{m}_{q,X}  \\
  &\qquad\times\big\|f\big\|_{L_p(\Omega; X)},
\end{split}
\end{equation}
for all $f\in L_p(\Omega;X)$.
\end{theorem}

By choosing $p=q$, $\beta_0=\frac{1}{q\mathfrak{m}_{q,X}^q}$, $T_{\beta_0}=\beta_0B$, where $B=q^2(\mathfrak{m}_{q,X})^{2q+1}(1+\log(q)+q\log(\mathfrak{m}_{q,X}))$, in Corollary \ref{coro:main 2} we can estimate the constant in \eqref{Xu's estimate} in terms of the martingale cotype constant $\mathfrak{m}_{q,X}$ and exponent $q$ as follows:
\begin{equation}\label{Xu's new estimate}
\begin{split}
  &\lesssim\beta_0^{-2}B\mathfrak{m}_{q,X}\big\|f\big\|_{L_p(\Omega; X)}  \\
  &=(q\mathfrak{m}_{q,X}^q)^2B\mathfrak{m}_{q,X}\big\|f\big\|_{L_p(\Omega; X)}.
\end{split}  
\end{equation}

\begin{remark}
In Corollary \ref{coro:main 2} we assume that $\{T_t\}_{t>0}$ is as symmetric diffusion semigroup on $(\Omega,\mathcal{A},\mu)$. Notice that (since any positive contraction is regular) if $\{T_t\}_{t>0}$ is a symmetric diffusion semigroup then $\{T_t\}_{t>0}$ is a strongly continuous semigroup of regular operators on $L_p(\Omega)$.
\end{remark}

We see that the bound \eqref{LPS eq:first approach} that we get by the first approach in the special case $p=q$, $k=1$ is somewhat better than what we get in \eqref{Xu's new estimate} by the second approach in the same case. We haven't traced the depedence that the second approach would give for $p\ne q$, since this would require the analyticity bounds of Corollary \ref{coro:main 2} for $p\ne q$, while our considerations in Section \ref{section:4} are most naturally adapted to the case $p=q$.

\subsection*{Acknowledgements} 
Both authors were supported by the Academy of Finland through project Nos. 314829 (``Frontiers of singular integrals'') and 346314 (``Finnish Centre of Excellence in Randomness and Structures''). Also, the second author would like to thank the Foundation for Education and European Culture (Founders Nicos and Lydia Tricha), Greece, for their financial support.
We would like to thank the anonymous referee for careful reading and constructive comments that improved the presentation.

This research was carried out while both authors were working at the University of Helsinki.

\end{document}